\documentclass{amsart}  
\usepackage{amsmath}
\usepackage{amsfonts,amssymb}
\usepackage{tikz}
\usepackage{pgfplots}
\pgfplotsset{compat=1.10}

\pgfplotsset{grid style={dotted, gray!80!white}}

\usepackage{graphicx,xcolor}
\usepackage{bbm,enumerate}
\usepackage{esint}
\usepackage{colortbl}

\usepackage{mathtools}      
\mathtoolsset{showonlyrefs} 

\numberwithin{equation}{section}
\newtheorem{theorem}{Theorem}[section]

\newtheorem{lemma}[theorem]{Lemma}

\newtheorem{proposition}[theorem]{Proposition}

\theoremstyle{remark}
\newtheorem{remark}[theorem]{Remark}

\theoremstyle{definition}
\newtheorem{definition}[theorem]{Definition}

\newcommand{\dist}{\mathrm{dist}}
\newcommand{\supp}{\mathrm{supp\,}}

\newcommand{\N}{\mathbb{N}}
\newcommand{\R}{\mathbb{R}}
\newcommand{\Z}{\mathbb{Z}}
\newcommand{\T}{\mathbb{T}}

\renewcommand{\hat}{\,\widehat}

\usepackage{scalerel,stackengine}
\stackMath
\newcommand\reallywidehat[1]{%
\savestack{\tmpbox}{\stretchto{%
  \scaleto{%
    \scalerel*[\widthof{\ensuremath{#1}}]{\kern-.6pt\bigwedge\kern-.6pt}%
    {\rule[-\textheight/2]{1ex}{\textheight}}
  }{\textheight}%
}{0.5ex}}%
\stackon[1pt]{#1}{\tmpbox}%
}
\usepackage{mathrsfs}

\title[Convergence over fractals for the periodic Schr\"odinger equation]{Convergence over fractals for the periodic Schr\"odinger equation}

\author{Daniel Eceizabarrena}
\address{Department of Mathematics and Statistics, University of Massachusetts Amherst, Amherst, MA 01003, USA}
\email{eceizabarrena@math.umass.edu}

\author{Renato Luc\`a}
\address{BCAM - Basque Center for Applied Mathematics, 48009 Bilbao, Spain and Ikerbasque, Basque Foundation
for Science, 48011 Bilbao, Spain.}
\email{rluca@bcamath.org}

\begin{document}

\begin{abstract}
We consider a fractal refinement of Carleson's problem for pointwise convergence of solutions to the periodic Schr\"odinger equation to their initial datum. For $\alpha \in (0,d]$ and 
\[ s < \frac{d}{2(d+1)} (d + 1 - \alpha), \]
we find a function in $H^s(\mathbb{T}^d)$ whose corresponding solution diverges in the limit $t \to 0$ on a set with strictly positive
$\alpha$-Hausdorff measure. We conjecture this regularity threshold to be optimal. We also prove that 
\[ s > \frac{d}{2(d+2)}\left(  d+2-\alpha \right) \]
 is sufficient for the solution corresponding to every datum in $H^s(\mathbb T^d)$ to converge to such datum $\alpha$-almost everywhere. 
\end{abstract}

\maketitle

\section{Introduction}

We consider the problem of determining when  the linear Schr\"odinger flow $e^{it\Delta} f(x)$ converges pointwise almost everywhere to the initial datum $f$ for every $f \in H^{s}(\Omega^d)$.
%
%
In \cite{Carleson1980}, Carleson proved for $\Omega = \R$ that this property holds for all $f \in H^{1/4}(\R)$, and  Dahlberg and Kenig \cite{DahlbergKenig1982} built counterexamples $f \in H^s(\mathbb R^d)$ for every $d \geq 1$ and every $s <1/4$ for which convergence fails, showing that Carleson's one-dimensional result was sharp. The considerably harder higher dimensional problem has been studied by many authors 
\cite{Cowling1983, Carbery1985, Sjolin1987, Vega1988, Bourgain1992, MoyuaVargasVega1999, TaoVargas2000_1, TaoVargas2000_2, Tao2003, Lee2006, Bourgain2013, LucaRogers2019, DemeterGuo2016, LucaRogers2017, DuGuthLiZhang2018} and recently solved up to endpoints in the  contributions of Bourgain~\cite{Bourgain2016}, who proved the necessity of  
$s \geq \frac{d}{2(d+1)}$ (see also \cite{LucaRogers2019_2} for an alternative proof and \cite{Pierce2020} for a nice detailed version of Bourgain's argument) and of Du, Guth and Li~\cite{DuGuthLi2017} and of Du and Zhang \cite{DuZhang2019}, who proved 
the sufficiency of $s > \frac{d}{2(d+1)}$ in dimensions~$d = 2$ and~$d \geq 3$ respectively.

In the periodic setting, namely when $\Omega = \T = \R / 2 \pi \Z$ is the torus, much less is known. For $d=1$, Moyua and Vega \cite{MoyuaVega2008} observed using Strichartz estimates \cite{Bourgain1993,BourgainDemeter2015} that 
$s > 1/3$ is sufficient for almost everywhere pointwise convergence and that $s \geq \frac14$ is necessary. In fact all the counterexamples on $\R^d$ can be resettled on $\T^d$
to show the necessity of $s \geq \frac{d}{2(d+1)}$; see \cite{CompaanLucaStaffilani2020}. In higher dimensions, again Strichartz estimates imply the sufficiency 
of $s > \frac{d}{d+2}$ as shown in \cite{WangZhang2019,CompaanLucaStaffilani2020}.    
To the knowledge of the authors, in the periodic setting the problem is open when $s\in \left[ \frac{d}{2(d+1)},  \frac{d}{d+2} \right]$. 
The reason this problem is harder, at least superficially, is that in the periodic setting there are more resonances to be controlled. The same happens
for the Strichartz estimates, which are considerably harder to prove on $\T^d$ than on $\R^d$.

Given $\alpha \in [0,d]$, a natural refinement of the  problem is to identify the minimal regularity $s$ such that one has $\lim_{t \to 0} e^{it \Delta} f(x) = f(x)$  $\alpha$-almost everywhere, that is to say, for all $x$ except maybe on a set of zero $\alpha$-Hausdorff measure, for all $f \in H^{s}(\Omega^d)$.
Of course, the first thing that one needs to care about is that the solution should be well defined $\alpha$-almost everywhere. We will show that this happens when $s > (d-\alpha)/2$. This is a natural threshold, since if $s < (d-\alpha)/2$ there exist functions in $H^s(\mathbb R^d)$ that are not well defined on sets of 
dimension $\alpha$; see~\cite{Zubrinic2002}. 

This refined problem was initiated for $\Omega = \R$ in \cite{SjogrenSjolin1989}, and solved for $\alpha \in [0,d/2]$ in \cite{BarceloBennettCarberyRogers2011}, where the condition for the well posedness of the problem $s > (d-\alpha)/2$ was surprisingly proved to be also sufficient for $\alpha$-almost everywhere convergence of the solution to the initial data. 
When $\alpha \in (d/2, d]$, the best result so far was shown in \cite{DuZhang2019}, where $\alpha$-almost everywhere pointwise convergence has been proved to hold
for all initial data $f \in H^{s}(\R^d)$ as long as
$$
s > \frac{d}{2(d+1)} (d + 1 - \alpha)  .
$$
We prove here that, excluding the endpoint, this is a necessary condition in the periodic case $\Omega = \T$.

\begin{theorem}\label{MainThm}
Let $d\ge 1$, $0 <  \alpha \leq d$ and $s \geq 0$ such that
\begin{equation}\label{INTVAL}
s < \frac{d}{2(d+1)} (d + 1 - \alpha).
\end{equation}
Then, 
there exist $f \in H^{s}(\mathbb{T}^{d})$ and $\Gamma \subset \mathbb{T}^{d}$ such that $\operatorname{dim}_{\mathcal H} \Gamma = \alpha$ and 
\begin{equation}\label{DivergentSolution}
\limsup_{t \to 0} | e^{it \Delta} f(x) | = \infty \qquad \text{ for all } x \in \Gamma.
\end{equation}
Moreover, when $\alpha = d$, $\mathcal H^d(\Gamma) >0$.
\end{theorem}

\begin{remark}\label{RemarkOnPositiveMeasureOfGamma}
That $\Gamma$ can be chosen to satisfy $\mathcal H^\alpha (\Gamma) >0$ when $\alpha < d$ can be deduced at once from the result of the theorem. Indeed, since \eqref{INTVAL} can be rewritten as 
\begin{equation}
\alpha < d+1 - \frac{2(d+1)}{d}\, s,
\end{equation}
there exists $\alpha'< d$ such that $\alpha < \alpha' < d+1 - \frac{2(d+1)}{d}\, s$. Hence, Theorem~\ref{MainThm} applies and yields a function $f \in H^s(\mathbb T^d)$ and a set $\Gamma'$ with $\operatorname{dim}_{\mathcal H} \Gamma' = \alpha'$ such that \eqref{DivergentSolution} holds for all $x \in \Gamma'$, and satisfies $\mathcal H^\alpha(\Gamma') = +\infty$. 
\end{remark}

Since periodic solutions exhibit more resonances, it is not surprising that more counterexamples are available, which makes it easier to find initial data for which  
the pointwise convergence fails. In the spirit of what was said above, this is why it is 
harder to prove the analogue of Theorem 
\ref{MainThm} on $\R^d$. In fact, this statement would fail for $\alpha \leq d/2$, as we may see comparing it with the sharp convergence result 
proved in~\cite{BarceloBennettCarberyRogers2011}.
In other words, Theorem \ref{MainThm} shows that it is impossible to extend the sharp convergence result of
that work 
to the periodic setting.

We also prove here a sufficient condition for $\alpha$-almost everywhere convergence on $\mathbb T^d$, for which we will need to work with $\alpha$-dimensional measures.
\begin{definition}\label{DefADim}
Let $0 < \alpha \leq d$.  We say that the measure $\mu$ is $\alpha$-dimensional if it is a positive Borel measure
and satisfies
\begin{equation}\label{C_Alpha}
c_{\alpha}(\mu) = \sup_{\substack{x \in \mathbb T^d \\ r > 0}}
\frac{\mu(B(x,r))}{r^{\alpha}} < \infty.
\end{equation}
\end{definition} 
It is well-known that sufficient conditions for convergence are deduced from $L^p$ estimates for the Schr\"odinger maximal function, which is the main content of the following theorem:
\begin{theorem}\label{Thm2}
Let $d\ge 1$, $0 <  \alpha \leq d$ and 
\begin{equation} \label{sValue}
s > \frac{d}{2(d+2)} (d + 2 - \alpha).
\end{equation}
Then
\begin{equation}\label{MaximalEstimate}
\left\| \sup_{0< t < 1} | e^{it\Delta} f | \right\|_{L^{\frac{2(d+2)}{d}}(\mathbb T^d, d \mu)} \lesssim
 c_{\alpha}(\mu)^{\frac{d}{{2(d+2)}}}
\| f \|_{H^{s}(\T^d)}
\end{equation}
for every $\alpha$-dimensional measure $\mu$. Consequently, 
\begin{equation}
\lim_{t \to 0} e^{it\Delta}f(x) = f(x) \qquad \alpha\text{-almost everywhere}, \qquad \forall f \in H^s(\mathbb T^d).
\end{equation}
\end{theorem}

It looks reasonable to us to conjecture the optimality of \eqref{INTVAL} in the cases
\begin{equation}\label{Conj}
\begin{array}{llll}
\bullet & \Omega = \R & \mbox{and} & \alpha \in (d/2, d], 
 \\
\bullet & \Omega = \T & \mbox{and} & \alpha \in (0,d]. 
\end{array}
\end{equation}
Also, it is usual to reformulate all the previous results and conjectures in terms of the divergence set of $f  :\Omega^d \to \mathbb{C}$, 
\begin{equation*}
\mathcal{D}(f) = \Big\{\, 
x \in \Omega^d \ : \lim_{t\to 0} e^{it \Delta}f(x) \neq f(x)
\,\Big\},
\end{equation*}
and of its dimension; in other words, by looking for 
\begin{equation*}
\alpha_{\Omega^d}(s) = 
\sup_{f\in H^{s}(\Omega^d)} \dim_{H}\big( \mathcal{D}(f) \big), 
\end{equation*}
the largest possible Hausdorff dimension of the divergence sets of all initial data in $H^s(\Omega^d)$ for a fixed $s \geq 0$.
Taking into account that due to the Sobolev embeddings the divergence sets are empty for $s> d/2$, the results known so far are 
%
\begin{equation*}
  \begin{array}{ll}
\alpha_{\R^d}(s) =  d,  &  s  <   \frac{d}{2(d+1)}, \quad  \text{\cite{DahlbergKenig1982, Bourgain2016}},
  \\ & \\ 
\alpha_{\R^d}(s) \in   \left[d+\frac{d}{d-1}-\frac{2(d+1)}{d-1} s,  \, d+1 - \frac{2(d+1)}{d} s \right] ,  &  s \in \left[ \frac{d}{2(d+1)},   \frac{d+1}{8} \right), \quad \text{\cite{LucaRogers2019_2, DuZhang2019}},
  \\  & \\ 
\alpha_{\R^d}(s) \in \left[  d+1-\frac{2(d+2)}{d}s, \, d+1 - \frac{2(d+1)}{d} s \right],  & s \in \left[ \frac{d+1}{8},  \frac{d}{4} \right), \quad \text{\cite{LucaRogers2017, DuZhang2019}},  
  \\ & \\ 
\alpha_{\R^d}(s) = d-2s , & s \in \left[ \frac{d}{4}, \frac{d}{2} \right], \quad \text{\cite{BarceloBennettCarberyRogers2011}}. 
\end{array}
\end{equation*}
In the periodic case, Theorems~\ref{MainThm} and \ref{Thm2} that we prove here can be rewritten as
\begin{equation}\label{FractalResult}
\begin{array}{ll}
\alpha_{\T^d}(s)=  d,  &  \quad s  <   \frac{d}{2(d+1)}
\\ & \\
\alpha_{\T^d}(s) \geq  d+1 - \frac{2(d+1)}{d} s,  &   \quad  s  \in \left[ \frac{d}{2(d+1)},  \frac{d}{2} \right],
 \\ 
 &
 \\
\alpha_{\T^d}(s) \leq  d+2 - \frac{2(d+2)}{d} s , &   \quad s \in \left(\frac{d}{d+2}, \frac{d}{2}\right],
 \end{array}
\end{equation}
and the conjecture \eqref{Conj} becomes 
\begin{equation}
\begin{array}{ll}
\alpha_{\R^d}(s) = d+1 - \frac{2(d+1)}{d} s,  &  \quad s \in \left[ \frac{d}{2(d+1)}, \frac{d}{4} \right) 
\\ & \\
\alpha_{\T^d}(s) = d+1 - \frac{2(d+1)}{d} s,  &   \quad  s \in \left[ \frac{d}{2(d+1)}, \frac{d}{2} \right].
 \end{array}
\end{equation}

The proof of Theorem~\ref{MainThm} elaborates on the resonances of the simple family of solutions
\begin{equation}\label{fdhsjkfskl}
e^{it\Delta} f(x) \qquad \text{ such that } \qquad f(x) = \sum_{k \in \mathbb{Z}^{d}, \,  |k_\ell| \leq N} \,  e^{ik \cdot  x },
\end{equation}
when the frequency parameter $N$ is very large. Note that $f$ is nothing more than the 
Dirichlet kernel. This family of solutions was used by Bourgain \cite{Bourgain1993} 
as a counterexample to show that periodic Strichartz estimates necessarily require a logarithmic derivative loss, that is, that when $d=1$ and $f$ as in \eqref{fdhsjkfskl}, one has
\begin{equation}
\int_{\T^2} |e^{it\Delta} f(x)|^6 dx \,  dt \gtrsim \ln (N) \,  \| f \|^6_{L^{2}(\mathbb T)}. 
\end{equation}
In the same spirit, and also in $d=1$, Moyua and Vega \cite{MoyuaVega2008} used \eqref{fdhsjkfskl} to disprove~$L^p(\T)$ estimates for the maximal Schr\"odinger operator $\sup_{t} |e^{it\Delta}f(x)|$ when $f \in H^{s}(\T)$ and $s < 1/4$. The key mathematical tool to understanding the behavior of the solutions~\eqref{fdhsjkfskl} for $N \gg1$ is the quadratic Gauss sums which we introduce later in~\eqref{GaussSum}. 

From a physical point of view, it was shown in \cite{BerryKlein1996} that \eqref{fdhsjkfskl} can be considered as a model for the Talbot effect in optics. For an introduction in the topic, we refer the reader to \cite{BerryMarzoliSchleich2001}. Fractal patterns were studied in \cite{BerryKlein1996}, so it is not surprising that \eqref{fdhsjkfskl} is a source of fractal counterexamples to several problems. Much work has been done in the context of dispersive PDEs related to the Talbot effect, we refer to \cite{ErdoganTzirakis2016} 
for a nice introduction. 
In particular, due to this connection to dispersive PDEs, the Talbot effect was discovered to govern the evolution of polygonal vortex filaments \cite{delaHozVega2014,delaHozVega2018,BanicaVega2018,delaHozKumarVega2019} and related to the celebrated Riemann non-differentiable function \cite{BanicaVega2022} and to its previously unexplored geometric interpretation \cite{Eceizabarrena2019,Eceizabarrena2020,Eceizabarrena2021}.

It is worth mentioning that, in $\mathbb{R}^n$, solutions exhibiting periodic constructive interference patterns on small times scales like \eqref{fdhsjkfskl} were used in \cite{BarceloBennettCarveryRuizVilela2007} as counterexamples to disprove dispersive estimates related to the Schr\"odinger operator. After a suitable modulation \cite{LucaRogers2017} or 
a pseudo-conformal transformation \cite{DemeterGuo2016}, these solutions can be also used to prove non trivial (although not optimal) necessary conditions for local $L^p$ bounds for the maximal Schr\"odinger operator. Comparing these results with Theorem \ref{MainThm}, we note again that the necessary conditions in the periodic setting proved in this article are stronger due to the presence of more resonances.   

Regarding Theorem~\ref{Thm2}, the estimate \eqref{MaximalEstimate} was proved when $\mu$ is the Lebesgue measure in \cite{MoyuaVega2008} for $d=1$, in \cite{WangZhang2019} for $d=2$ and in \cite{CompaanLucaStaffilani2020} for $d \geq 3$, using periodic Strichartz estimates \cite{BourgainDemeter2015}. We will be able to promote these to $\alpha$-dimensional measures by analyzing their behavior when convoluting against the Dirichlet kernel.

The structure of this article is as follows. We begin by setting notation, elemental tools and auxiliary well-known results in Section~\ref{SECTION_Setup}. In Sections~\ref{SECTION_MainTheorem} and \ref{SECTION_Measure} we prove Theorem~\ref{MainThm}. We first build an initial datum $f$ and a set $\Gamma$ where the divergence property \eqref{DivergentSolution} holds, and then we prove that the dimension of the set $\Gamma$ is $\alpha$. In Section~\ref{SectionThm2} we prove Theorem~\ref{Thm2}. We conclude with Appendix~\ref{Appendix}, where we prove that the problem is well posed, that is, that the solution $e^{it\Delta}f(x)$ is well defined $\alpha$-almost everywhere for $f \in H^s(\mathbb T^d)$ as long as $d-\alpha < 2s$.

\section{Preliminaries}\label{SECTION_Setup}

Given $f : \T^d \to \mathbb{C}$ and $k \in \Z^d$, let 
\begin{equation}
\hat{f}(k) = \frac{1}{(2 \pi)^{d}} \int_{\mathbb{T}^d} f(x)\,  e^{-i k \cdot x}  \, dx
\end{equation}
be its Fourier coefficients. Define the solution to the linear Schr\"odinger equation as the pointwise limit of the partial Fourier sums in cubes, namely 
\begin{equation}\label{SolSchr}
e^{it\Delta} f(x) = \lim_{N \to \infty} S_{N}(t)f(x) ,
\end{equation}
where
\begin{equation}\label{rt}
S_{N}(t)f(x) =
\sum_{\substack{ k \in \mathbb{Z}^d \\ |k_\ell| \leq N \\ \ell=1, \ldots, d} }   \hat{f}(k)\, e^{ ik\cdot x -  i  | k |^{2} t } .
\end{equation}
The limit \eqref{SolSchr} is usually taken with respect to the $L^2$ norm, but here we take all limits pointwise, at each point $x$ that they exist. 
When $f \in L^2$, it is known that the limit exists pointwise for almost every $x \in \T$ and that it coincides with the $L^2$ limit. The result when $d = 1$ is due to Carleson \cite{Carleson1966},
whose proof extends to higher dimensions as proved, for instance, in~\cite{Fefferman1971}.  
In Appendix~\ref{Appendix}, we show that this limit exists $\alpha$-almost everywhere for every $f \in H^{s}$ with $s \in (0, d/2]$, as long as $\alpha>d-2s$.
This can be regarded as a refinement of Carleson's result, although it does not recover it. 

%

The Dirichlet kernel 
\begin{equation}\label{DirichletKernel}
D_N(x) = \prod_{\ell = 1}^{d} d_N(x_\ell),  \quad d_N(x_\ell) = \sum_{\substack{ k_\ell \in \mathbb{Z} \\ |k_\ell | \leq N}} e^{i k_\ell x_\ell }
\end{equation}
allows for alternative representations of the Fourier truncations above given that the solution \eqref{SolSchr} is well defined, since
\begin{equation}
S_{N}(t) f(x) = D_N * e^{it\Delta} f(x) \quad \text{ and } \quad S_{N}(0) f(x) = D_N * f(x).
\end{equation}
%
%
%
%
%
%
%

Also, the generalized quadratic Gauss sums will be the key to analyzing the behavior of the solution in the proof of Theorem~\ref{MainThm}.
It is well-known (see, for instance, \cite[Appendix A]{delaHozVega2014}) that for any $q \in \mathbb{N}$ and $p,r \in \mathbb{Z}$ such that $\operatorname{gcd}(r,q) = 1$ we have 
\begin{equation}\label{GaussSum}
 \left| \sum_{ k = 0}^{ q-1  } e^{ 2 \pi  i \frac{r k^2 + p k }{q } } \right|  = \left\{ \begin{array}{ll}
 \sqrt{q}, & \text{ if } q \text{ is odd,} \\
 \sqrt{2q}, & \text{ if } q \equiv 0 \text{ (mod }4) \text{ and } p \equiv 0 \text{ (mod  } 2), \\
  \sqrt{2q}, & \text{ if } q \equiv 2 \text{ (mod }4) \text{ and } p \equiv 1 \text{ (mod  } 2), \\
  0, & \text{ otherwise.}
 \end{array} \right.
\end{equation}

We will be dealing with partial Gauss sums too, which we will manage by the following Van der Corput estimates. Hereafter, we say that $I$ is an integer interval if its boundary points are natural numbers, 
namely $I = [K_1, K_2]$ with $K_1, K_2 \in \N$.
We denote 
by $|I|$ the length of~$I$ and we denote by $L(I) = K_1$ its left boundary and by $R(I)=K_2$ its right boundary. In general, we will use the short notation $k \in I$ for $k \in \mathbb{N} \cap I $. 

The first lemma corresponds to the Van der Corput second derivative test, which will be useful to show that incomplete Gauss sums will not be of a larger order than the complete ones.  
We refer the reader to \cite[Theorem 2.2]{GrahamKolesnik1991} for a proof. 
\begin{lemma}
\label{VDCTest2}
Let $I$ be an integer interval and $f(x)$ be a real valued function with two continuous derivatives on $I$. Assume that there exist $M >0$ and $\alpha \geq 1$ such that 
\begin{equation}
M \leq | f''(x) | \leq \alpha\, M, \qquad \forall x \in I.
\end{equation} 
Then, 
\begin{equation}
\left| \sum_{k \in I } e^{2 \pi i f(k)} \right| \lesssim \alpha |I| M^{1/2} + M^{-1/2}.
\end{equation}
In particular, if $M = 1/q$ and $|I| \leq q$, 
\begin{equation}
\left| \sum_{k \in I } e^{2 \pi i f(k)} \right| \lesssim  \alpha \sqrt{q}
\end{equation}
\end{lemma}
We will also use Van der Corput's first derivative test, whose proof can be found in \cite[Theorem 2.1]{GrahamKolesnik1991}.
\begin{lemma}
\label{VDCTest1}
Let $I$ be an integer interval. Let $f(x)$ be a real valued function with a monotonic continuous derivative such that
$$
\min_{x \in I} \, \dist \left( f'(x) , \mathbb{Z}  \right) \geq \kappa > 0.
$$ 
Then 
$$
\left| \sum_{k \in I } e^{2 \pi i f(k)} \right| \leq C \kappa^{-1}.
$$
\end{lemma}

Finally, we present Abel's inequality in the following lemma.
\begin{lemma}
\label{DirichletSumLemma}
Let $I$ be an integer interval. Let $ a_k \geq 0$ be a sequence of real numbers and $b_k$ a sequence of complex numbers such that 
\begin{enumerate}
\item $a_{k+1} \leq a_k,$ 
\item
$\left| \sum_{k \in I' } b_k  \right| \leq \mathcal{C} , \quad  \mbox{$\forall$ integer interval $I' \subseteq I$}$.
\end{enumerate}
Then, 
\begin{equation}\label{DirichletTest}
\left| \sum_{k \in I'} a_k b_k \right| \leq  \mathcal{C} a_{L(I')} , \quad  \mbox{$\forall$ integer interval $I' \subseteq I$}.
\end{equation}
If (1) is replaced with $a_{k+1} \geq  a_k$, then 
$$
\left| \sum_{k \in I'} a_k b_k \right| \leq  \mathcal{C} a_{R(I')}, \quad  \mbox{$\forall$ integer interval $I' \subseteq I$}.
$$
\end{lemma}
\begin{proof}
Let $I' = [K_1, K_2]$.
We only consider the case $a_{k+1} \leq a_k$, since the case $a_{k+1} \geq a_k$ can be handled by changing variables
$k \leftrightarrow K_1 + K_2 - k$. 
The inequality~\eqref{DirichletTest} easily follows by the summation by parts formula
\begin{equation}
\sum_{k=K_1}^{K_2} a_k b_k = a_{K_2} B_{K_2} + \sum_{k=K_1}^{K_2-1} B_k(a_k - a_{k+1}), 
\quad \text{ where } \quad 
B_{k} = \sum_{\ell=K_1}^{k}b_\ell.
\end{equation}
 By \textit{(2)}, we have    $|B_{k}| \leq \mathcal{C}$ for every $k$, so  
\begin{equation}
\begin{split}
\left| \sum_{k=K_1}^{K_2} a_k b_k  \right| &  \leq \mathcal{C}\, |a_{K_2}| + \mathcal{C}\,  \sum_{k=K_1}^{K_2-1} |a_k - a_{k+1}|  = \mathcal{C} \left(  a_{K_2} + \sum_{k=K_1}^{K_2-1} (a_k - a_{k+1}) \right) \\
& = \mathcal{C} \left(  a_{K_2} + a_{K_1} - a_{K_2}  \right) = \mathcal{C}\, a_{K_1}.
\end{split}
\end{equation}  
\end{proof}

\section{Proof of Theorem \ref{MainThm}}\label{SECTION_MainTheorem}

We begin the proof of Theorem~\ref{MainThm} with a remark on the well-posedness of the problem. We said in the introduction, and we prove in Appendix~\ref{Appendix}, that the solution $e^{it\Delta}f(x)$ given as the pointwise limit in \eqref{SolSchr} is well-defined $\alpha$-almost everywhere for $f \in H^s(\mathbb T^d)$ as 
long as $d-\alpha < 2s$. Since $\alpha >0 $ implies 
\begin{equation}
\frac{d-\alpha}{2} < \frac{d}{2(d+1)}(d+1-\alpha), 
\end{equation}
we will assume throughout the proof that $s$ is such that
\begin{equation}\label{RestrictionForWellPosedness}
\frac{d-\alpha}{2} < s < \frac{d}{2(d+1)}(d+1-\alpha),
\end{equation}
so that it is precisely the case. For smaller values of $s$, the initial datum itself may not be well defined with respect to the $\alpha$-Hausdorff measure, in which case the problem is not well posed.


We are now ready to prove Theorem \ref{MainThm}.
Let 
\begin{equation}
s_{\alpha} = \frac{d}{2(d+1)} (d + 1 - \alpha) ,
\end{equation}
and define the initial datum  
\begin{equation}\label{TestinfFunctBennFinLLD}
f = \sum_{j \in \mathbb{N}}  f_j
\quad \text{ such that } \quad
f_j(x) = \lambda^{-j \left( s_{\alpha} + \frac{d}{2} - \delta \right) } \prod_{\ell =1}^{d} \, \, \sum_{n_\ell = \lambda^{j-1}}^{\lambda^j-1}   e^{i n_\ell  x_\ell },
\end{equation}
where $\lambda$ is an integer satisfying $\lambda \gg 1$ and $0 < \delta \ll 1$. Since
$\| f_j \|_{H^{s}(\T^d)} \simeq \lambda^{-j \left( s_{\alpha}  - \delta -s \right) }$, we have $f \in H^{s}(\T^d)$ as long as $s < s_{\alpha} - \delta $.
Let $0 < \kappa  < 1$ and 
\begin{equation}\label{DefT_j}
T^j =   \left\{ \frac{2 \pi}{q} \, : \,  q \in  \mathbb{N} \cap \left[ \kappa\, \lambda^{j\frac{\alpha}{d+1}} , \lambda^{j\frac{\alpha}{d+1}} \right] ,\quad q \equiv 0 \text{ (mod }4) \right\}. 
\end{equation}
For each time $t \in T^j$, denote by $q(t) = 2 \pi / t$ the corresponding integer, and let 
\begin{equation}\label{Def:X_j}
X_t^j = \bigcup_{\substack{p \in \mathbb{Z}^d \cap \left[ \frac{q(t)}{4}, \frac{q(t)}{2} \right]^d \\ p_i \equiv 0 \text{ (mod 2)}, \, \forall i =1, \dots, d } }   \left\{ \,  2\pi \left( \frac{ p}{q(t)} + \varepsilon \right) \, : \, \frac{1}{200 \lambda^j} \leq \varepsilon_i \leq \frac{1}{100\lambda^j}  \,\, \forall i = 1, \ldots, d \, \right\} ,
\end{equation}
where we define $p = (p_1, \ldots, p_d) \in \mathbb{Z}^d$ and $ \varepsilon = (\varepsilon_1,\ldots, \varepsilon_d)\in\mathbb{R}^d$. 
Notice that $X_t^j$ is a union of translations of closed cubes that have side length and distance from the rational $p/q$ equal to $(200 \lambda)^{-j}$.   
The objective will be to show that the modulus of the solution is large at the times $t \in T^j$ over the sets $X_t^j$, for every $j \in \mathbb{N}$.
That will mean that if we define
\begin{equation}\label{Def:Gammaj}
\Gamma^j = \bigcup_{t \in T^j} X_t^j, \qquad j \in \mathbb{N},
\end{equation}
for every $x \in \Gamma^j$ there will be a time $t\in T^j$ such that the solution is large in $(x,t) \in (X^j_t,T^j)$. Thus, letting
\begin{equation}\label{Gamma}
\Gamma = \bigcap_{n \in \N} \bigcup_{j > n} \Gamma^j,
\end{equation} 
for every $x \in \Gamma$ there will exist a sequence $t_{j_k} \in T^{j_k}$ such that the solution $|e^{it_{j_k}\Delta}f(x)|$ is large for every $k$, leading to 
\begin{equation}\label{Final}
\limsup_{t \to 0} | e^{it \Delta}f(x) | = \infty, 
\qquad \forall x \in \Gamma \setminus D,
\end{equation}
where $D$ is a set with Hausdorff dimension strictly smaller than $\alpha$ that plays a technical role. 
For the theorem to follow, we will also require that $\Gamma$ has Hausdorff dimension $\alpha$. The Jarn\'ik-Besicovitch theorem suggests that it is indeed the case; we devote  Section~\ref{SECTION_Measure} to proving this last claim.


Let us be more precise now. To prove \eqref{Final}, it will be enough to show     
\begin{equation}\label{abc}
\begin{array}{llll} 
(i) & | S_{N}(t)  f_j ( x ) |
 \simeq \lambda^{j \delta } , &  (x, t) \in (X_t^j, T^j)  , &  \\
 (ii) & | S_{N}(t)  f_k ( x ) |
 \lesssim \lambda^{k \delta}, &  (x, t) \in (X_t^j, T^j), & k < j  , \\
(iii) & 
 | S_{N}(t)  f_k ( x ) |
\lesssim \lambda^{j \delta  - c (k-j)   } , & (x, t) \in (X_t^j, T^j), & k > j  ,   
\end{array}
\end{equation}
for $N > \lambda^j$, where $c >0$ is an absolute constant. 
Indeed, if $x \in \Gamma^j$, by \eqref{Def:Gammaj} there exists a time $t_{j}(x) \in T^{j}$ such that by \eqref{abc} we have 
\begin{equation}\label{abcBis}
\begin{array}{lll}
(i) & | S_{N}(t_j(x))  f_j ( x ) |
 \simeq \lambda^{j \delta}, &    
\\ 
(ii) & | S_{N}(t_j(x))  f_k ( x ) |
 \lesssim \lambda^{k \delta}, & k < j  , 
 \\
(iii) & | S_{N}(t_j(x))  f_k ( x ) |
\lesssim \lambda^{j \delta  - c(k - j)} , &  k > j     
\end{array}
\end{equation}
 for all $N > \lambda^j$. Since we have defined $\Gamma$ as the set of the points which belong to infinitely many $\Gamma^{j}$, for 
any $x \in \Gamma$ there exists an infinite subset $J(x)\subset\mathbb{N}$ with an associated 
sequence of times
$t_{j}(x) \in T^{j}$ for all $j\in J(x)$ such that \eqref{abcBis} is satisfied.
Recalling the definition~\eqref{TestinfFunctBennFinLLD} of $f$, by the triangle inequality and by \eqref{abcBis}(i) we have
\begin{equation}\nonumber
| S_{N}(t_j(x))  f(x) |  \, \geq \,  | S_{N}(t_j(x))   f_j(x) | - | A_1 | - | A_2 | \simeq \lambda^{j \delta} - | A_1 | - | A_2 | ,
\end{equation}
where
\begin{equation}\nonumber
A_1   =  \sum_{k < j}  S_{N}(t_j(x))  f_k(x) ,
\qquad
A_2   =   \sum_{k > j}  S_{N}(t_j(x)) f_k(x) .
\end{equation}
On the other hand, summing \eqref{abcBis}(ii) over $k =1, \ldots , j-1$ and \eqref{abcBis}(iii) over $k > j$, we see that
\begin{equation}\label{Reminders}
|A_1|  + |A_2| \ll \lambda^{j \delta}  .
\end{equation}
Indeed
$$|A_1| \leq \sum_{k =1}^{j-1}  |S_{N}(t_j(x))  f_k(x)| \lesssim \sum_{k =1}^{j-1} \lambda^{k \delta} \simeq \lambda^{(j-1) \delta} = \lambda^{-\delta} \lambda^{j \delta} $$
and
$$
|A_2| \leq \sum_{k \geq j+1}    | S_{N}(t_j(x))  f_k ( x ) | \lesssim \lambda^{j \delta } \sum_{k \geq j+1} \lambda^{  - c(k - j)}
\lesssim \lambda^{j \delta - c(j+1- j)  } = \lambda^{-c } \lambda^{j \delta},
$$
so \eqref{Reminders} holds taking $\lambda$ sufficiently large.
Thus, for any $x \in \Gamma $, 
there is a sequence of times $t_{j}(x) \in T_j$ for $j \in J(x)$
such that
\begin{equation}\label{ghdjksdjghjdskjdh}
| S_{N}(t_j(x)) f(x) | \simeq \lambda^{j \delta},
\end{equation}
for all $N > \lambda^j$. 

We need to take the limit $N \to \infty$ now. Define $T = \bigcup_{j \in \mathbb N}T^j$, and let $t \in T$. Then, the limit of \eqref{ghdjksdjghjdskjdh} exists for all $x \in \mathbb T^d \setminus D_t$, where we call $D_t$ precisely the set of points where the limit does not exist. We said, and we prove in Appendix~\ref{Appendix}, that $\mathcal H^\alpha(D_t) = 0$ as long as $\alpha > d-2s$, which means that $\dim_{\mathcal{H}}D_t \leq d-2s$. Define now $D = \bigcup_{t \in T} D_t$ the set of points where convergence fails for some $t \in T$. Since $T$ is a countable set, $\dim_{\mathcal H}D = \sup_{t \in T} \dim_{\mathcal H} D_t \leq d-2s$, and 
\begin{equation}
\lim_{N \to \infty} S_N(t)f(x) = e^{it\Delta}f(x), \qquad \forall x \in \mathbb T^d \setminus D, \quad \forall t \in T.
\end{equation}
Thus, taking $N \to \infty$ in \eqref{ghdjksdjghjdskjdh} yields
\begin{equation}\label{ReallyFinal}
| e^{i t_j(x)\Delta} f(x) | \simeq \lambda^{j \delta}, \qquad \forall x \in \Gamma \setminus D , \quad \forall j \in J(x), 
\end{equation} 
and then taking $j \to \infty$, we obtain
\begin{equation}\label{LimitIsInfinite}
\lim_{j \to \infty} | e^{i t_{j}(x) \Delta} f(x) |  = \infty, \qquad \forall x \in \Gamma \setminus D.
\end{equation}
From the definition of $T^j$, we have $t_j(x) \simeq \lambda^{-j\frac{\alpha}{d+1}}$, so $\lim_{j\to\infty}t_j(x) = 0$. This, together with \eqref{LimitIsInfinite} yields \eqref{Final}, what we wanted to prove, for $x \in \Gamma \setminus D$. And since we assumed in \eqref{RestrictionForWellPosedness} that $\dim_{\mathcal H} D \leq d-2s < \alpha = \dim_{\mathcal H}\Gamma$, we get $\operatorname{dim}_{\mathcal H}(\Gamma \setminus D) = \alpha$.

Hence, apart from checking that the Hausdorff dimension of $\Gamma$ is $\alpha$, which we do in Section~\ref{SECTION_Measure}, there remains only to establish \eqref{abc}.

\subsection*{Proof of \eqref{abc}(i)}
In fact we will prove that for $N > \lambda^j$
\begin{equation}\label{abc(i)extended}
| S_{N}(t) f_k ( x ) |
 \simeq \lambda^{k \delta - (j-k) \frac{d \alpha}{2(d+1)}}, \quad   (x, t) \in (X_t^j, T^j), \quad  j \left(  \frac{\alpha}{d+1} + \delta \right)  \leq k \leq j
\end{equation}
which implies \eqref{abc}(i) and \eqref{abc}(ii) in the range  $j \left( \frac{\alpha}{d+1} + \delta \right)  \leq k < j$.
Taking $\varepsilon \in \mathbb{R}^d$ with 
\[ |\varepsilon_\ell| \leq \frac{1}{100 \lambda^j} \]
 for every $\ell=1, \ldots, d$, we will show \eqref{abc(i)extended} at points
\begin{equation}\nonumber
t = \frac{2 \pi}{q}, \quad x = 2\pi \left(  \frac{ p}{q} + \varepsilon \right), 
\qquad p \in \mathbb{Z}^d, \ q \in \mathbb{N}, 
\end{equation}
such that $q \simeq \lambda^{j\frac{\alpha}{d+1}}$, $q \equiv 0 \text{ (mod 4})$, $q/4 \leq p_\ell \leq q/2$ and $p_\ell \equiv 0 \text{ (mod 2})$ for all $ \ell = 1, \ldots, d$. 
Since $k \leq j$, we have  
\begin{equation}\label{SolutionWhenKSmallerThanJ}
S_{N}\left( \frac{2 \pi }{q} \right) f_k \left(  2\pi \left( \frac{ p}{q} + \varepsilon \right) \right)  =  
\lambda^{-k \left( s_{\alpha} + \frac{d}{2} - \delta \right) }
\prod_{\ell=1}^d \sum_{n_\ell =\lambda^{k-1}}^{\lambda^{k}-1} e^{ 2 \pi i \left( n_\ell  \left( \frac{p_\ell}{q} + \varepsilon_\ell \right)  -  \frac{n_\ell^2}{q} \right)} ,
\end{equation}
%
so it will suffice to prove 
\begin{equation}\label{Plug4}
\left| \sum_{n_\ell =\lambda^{k-1}}^{\lambda^{k}-1} e^{ 2 \pi i \left( n_\ell  \left( \frac{p_\ell}{q} + \varepsilon_\ell \right)  -  \frac{n_\ell^2}{q}  \right)} \right|
\simeq \lambda^{k - j \frac{\alpha}{2(d+1)} }  ,
\end{equation}
for all $k \in \left( j  \left( \frac{\alpha}{d+1} + \delta \right) ,  j\right]$, because that way we get 
\begin{align}\label{FInalControfOrder}
& \left| S_{N}\left( \frac{2 \pi }{q} \right) 
 f_k \left( 2 \pi \, \left( \frac{p}{q} + \varepsilon \right) \right) \right|
 \simeq 
\lambda^{-k \left( s_{\alpha} + \frac{d}{2} - \delta \right)  }
 \lambda^{k d \left( 1 -   \frac{\alpha}{2(d+1)} \right) } \lambda^{- (j - k) d \frac{\alpha}{2(d+1)} } 
 \\ \nonumber
 &
 \quad \quad \quad
\simeq \lambda^{k \delta - k \left( s_{\alpha} - \frac{d}{2(d+1)} \left( d + 1 - \alpha \right)  \right)  } \lambda^{- (j - k) d \frac{\alpha}{2(d+1)} } 
\simeq 
\lambda^{k \delta}  \lambda^{- (j - k) d \frac{\alpha}{2(d+1)} }  , 
\end{align}
as claimed. 

Let us prove \eqref{Plug4}.  
Let $L_k$ and $R_k$ be respectively the smallest and largest integers that satisfy 
\begin{equation}\label{DefLk}
  \lambda^{k-1} \leq L_k q  \quad     \text{ and }  \,   \quad R_k q \leq  \lambda^k -1,
\end{equation}
which means that
\begin{equation}\label{RL}
L_k q  - \lambda^{k-1} < q  \quad \text{ and }  \quad \lambda^k  - R_k q  <   q + 1.  
\end{equation}
We change variables $n_l = L_k q + m_l q + r$. 
The new values $m_l = 0, \ldots, R_k - L_k – 1$ and $r_l = 0, \ldots, q-1$ account for the old $n_l = L_k q, \ldots, R_k q – 1$.
The corresponding phase in \eqref{Plug4} satisfies
\begin{equation}
 n_\ell \left( \frac{p_\ell}{q} + \varepsilon_\ell \right)  =  ( (L_k + m_\ell ) q + r_\ell) \left( \frac{p_\ell}{q} + \varepsilon_\ell \right) 
\in \mathbb{Z} 
+ \left( L_k + m_\ell \right) q \varepsilon_\ell +   r_\ell \left( \frac{p_\ell}{q} + \varepsilon_\ell \right), 
\end{equation}
and
\begin{equation}
 \frac{n_\ell^2}{q}  = \frac{((L_k + m_\ell) q + r_\ell )^2}{q}  
\in \mathbb{Z} +   \frac{r_\ell^2}{q} ,
\end{equation}
so we split the sum \eqref{Plug4} into 
\begin{align}\label{Plug3}
 \sum_{n_\ell =\lambda^{k-1}}^{\lambda^k - 1} e^{ 2 \pi i \left( n_\ell \left(  \frac{p_\ell}{q} + \varepsilon_\ell \right)  -  \frac{n_\ell^2}{q} \right)} 
&   = \!\!\!\!\!\!\!
  \sum_{m_{\ell} = 0}^{R_k - L_k -1} e^{2 \pi i (L_k + m_\ell ) q \varepsilon_\ell}
\sum_{ r_\ell = 0}^{q - 1}e^{ 2 \pi i \left( r_\ell   \left( \frac{p_\ell}{q} + \varepsilon_\ell \right) -  \frac{r_{\ell}^2}{q} \right)}
\\ \nonumber
&
 + \!\!\!\!\!\!\!  
\sum_{ n_\ell \in J_1 \cup J_2} e^{ 2 \pi i \left( n_\ell  \left( \frac{p_\ell}{q} + \varepsilon_\ell \right) -  n_{\ell}^2 \frac1q \right) } 
= I + II  ,
 \end{align}
where 
\begin{equation}
J_1 = \{ \lambda^{k-1}, \ldots, L_k q -1  \}, \quad
J_2 = \{ R_k q, \ldots, \lambda^k - 1 \}.
\end{equation}
In order to evaluate $|I|$ we notice that for $m_\ell = 0, \ldots, R_k - L_k - 1$ we have  
\begin{equation}\label{ReasonForTheChoiceOfEpsilon}
| (L_k + m_\ell) q \varepsilon_\ell | < R_k   q | \varepsilon_\ell | < \lambda^k  \frac{1}{100\lambda^j}
  \leq  \frac{1}{100},
 \end{equation}
where we used \eqref{DefLk} and $k\leq j$. This means that
 all the phases $2\pi (L_k + m_\ell) q \varepsilon_\ell$ are close to zero, so the corresponding exponentials are very close to the unity.  
 Thus
\begin{equation}\nonumber
|I| \simeq (R_k - L_k)  
\left| \sum_{ r_\ell= 0}^{ q-1  } e^{ 2 \pi  i \left( r_\ell \left( \frac{p_\ell}{q} + \varepsilon_\ell \right) -  \frac{r_\ell ^2}{q} \right)} \right| .
\end{equation}
The error in the phase with respect to a Gauss sum is at most $|\varepsilon| q \lesssim \lambda^{-j} \lambda^{j \frac{\alpha}{d+1} } \ll 1$. We show in the forthcoming Lemma \ref{ItRemainsLemma}  that in this situation, we have
\begin{equation}\label{Perturb}
\left| \sum_{ r_\ell= 0}^{ q - 1  } e^{ 2 \pi  i \left( r_\ell \left( \frac{p_\ell}{q} + \varepsilon_\ell \right) -  \frac{r_\ell ^2}{q} \right)} \right| \simeq \sqrt{q}.
\end{equation} 
On the other hand, by \eqref{DefLk} and $k > j\alpha/(d+1)$,  recalling that $q \simeq \lambda^{j\frac{\alpha}{d+1}}$, we have 
$$
R_k - L_k \simeq \frac{\lambda^{k} - \lambda^{k-1}}{q} \simeq \frac{\lambda^k}{q}.
$$
Thus, we arrive at
\begin{equation}\label{Plug1}
|I| \simeq \frac{\lambda^k}{q} \sqrt{q} \simeq \frac{\lambda^k}{\sqrt{q}} \simeq   \lambda^{k - j \frac{\alpha}{2(d+1)} }.
\end{equation}

The estimate for $|II|$ follows from the Van der Corput second derivative test in Lemma \ref{VDCTest2}. Indeed, by \eqref{RL} we have $|J_1|, |J_2| \leq q$, and the second derivative of the phase of $II$ is $-2/q$, so
\begin{equation}\label{Plug2}
|II|  \lesssim \sqrt{q} \lesssim \lambda^{j  \frac{\alpha}{2(d+1)}  } \ll \lambda^{k - j \frac{\alpha}{2(d+1)} }  ,
\end{equation}
where the last inequality follows from $j  \left( \frac{\alpha}{d+1} + \delta \right)  < k$ and $\lambda \gg 1$.
Plugging \eqref{Plug1}  and  \eqref{Plug2} into \eqref{Plug3} we obtain \eqref{Plug4} and the proof of \eqref{abc(i)extended} is concluded. 

There remains to prove \eqref{Perturb}, which we do in the following lemma.
\begin{lemma}\label{ItRemainsLemma}
Let $q\in\mathbb N$, $p \in \mathbb Z$ and $\varepsilon \in \mathbb{R}$ such that $|\varepsilon|\, q \ll 1$, $q \equiv 0 \text{ (mod 4)}$ and $p \equiv 0 \text{ (mod 2)}$. Then
\begin{equation}\label{GS}
\left| \sum_{ r = 0}^{ q - 1  } e^{ 2 \pi  i \left( r \left( \frac{p}{q} + \varepsilon \right) -  \frac{r^2}{q} \right)} \right| \simeq \sqrt{q}  ,
\end{equation}
\end{lemma} 
\begin{proof}
Since the object under study is a perturbation of a Gauss sum in the case in which its value is $\simeq \sqrt{q}$ (see \eqref{GaussSum}), it suffices to show that the error introduced by this perturbation is of order smaller than $\sqrt{q}$. Write   
\begin{align}\label{ItRemains}
&  
\left|    \sum_{ r = 0}^{q - 1} e^{ 2 \pi  i \left( r \left( \frac{p}{q} + \varepsilon \right) - \frac{ r^2}{q} \right)} 
 -
 \sum_{ r = 0}^{ q - 1 } e^{ 2 \pi  i   \frac{ r  p - r^2}{q}}  
 \right|
\\ \nonumber
& 
\quad \quad
=
\left|   \sum_{ r = 0}^{ q - 1  } e^{ 2 \pi  i \frac{ r  p - r^2}{q}} \left( e^{2 \pi i r \varepsilon} - 1 \right)   \right|
\leq I + II,
\end{align}
where 
\begin{equation}
I = \left| \sum_{ r = 0}^{ q - 1  } e^{ 2 \pi  i \frac{ r  p - r^2}{q}} \left( \cos\left( 2 \pi r \varepsilon\right) - 1 \right) \right|, \quad II = \left| \sum_{ r = 0}^{ q - 1 } 
e^{ 2 \pi  i \frac{ r  p - r^2}{q}}  \sin \left( 2 \pi  r \varepsilon\right)   \right|
\end{equation}
and we have used $e^{2 \pi i r \varepsilon} = \cos \left( 2 \pi  r \varepsilon\right) + i \sin \left( 2 \pi  r \varepsilon\right)$ and 
triangle inequality.  Then, since $|\varepsilon| \, q \ll 1$, we can use the second part of Lemma \ref{DirichletSumLemma} with  $a_r = 1 - \cos \left( 2 \pi  r\varepsilon \right)$, 
$b_r = e^{ 2 \pi  i \frac{r\, p - r^2}{q} } $ and $\mathcal{C} = \sqrt{q}$ (which is allowed by Lemma~\ref{VDCTest2}),
so that
\begin{equation}
I \lesssim \sqrt{q} \, a_{q-1} = \sqrt{q}  \left( 1 - \cos \left( 2 \pi  (q - 1) \varepsilon\right) \right)   \lesssim \sqrt{q} \, q^2 |\varepsilon|^2  \ll \sqrt{q}.
\end{equation}
Similarly, if we choose $a_r = \frac{\varepsilon}{|\varepsilon|} \sin \left( 2 \pi  r \varepsilon\right)$ instead, then
\begin{equation}
II \lesssim \sqrt{q} \, a_{q-1} = \sqrt{q}\,  |\sin \left( 2 \pi  (q - 1) \varepsilon\right)| \lesssim \,  \sqrt{q}  \, q  |\varepsilon| \ll \sqrt{q}. 
\end{equation}   
\end{proof}

\subsection*{Proof of \eqref{abc}(ii)}
Since we already proved \eqref{abc(i)extended}, we only need to show \eqref{abc}(ii) for $1 \leq k \leq j \left( \frac{\alpha}{d+1} + \delta \right)$. 
In fact, we will prove the stronger
\begin{equation}\label{abc(i)stronger}
| S_{N}(t)  f_k ( x ) |
 \lesssim \lambda^{-ck}, \quad   (x, t) \in (X_t^j, T^j), \quad  1 \leq k \leq j \left( \frac{\alpha}{d+1} + \delta \right),
\end{equation}
for some $c > 0$ and all $N > \lambda^j$. As in the previous case, we work with
\begin{equation}\nonumber
t = \frac{2 \pi}{q}, \quad x = 2\pi \,  \left(\frac{ p}{q} + \varepsilon \right), \qquad q \in \mathbb N, \, p \in \mathbb Z^d, \, \epsilon \in \mathbb R^d
\end{equation}
such that $q \equiv 0 \text{ (mod 4)}$,  $q \simeq \lambda^{j\frac{\alpha}{d+1}}$,  $p_\ell \equiv 0 \text{ (mod 2)}$, $q/4 \leq p_\ell \leq q/2$ and  
\[ |\epsilon_\ell| \leq \frac{1}{100\lambda^j} \]
for all $\ell = 1, \ldots, d$.

Let us first consider the case $k \leq j \left( \frac{\alpha}{d+1} - \delta \right)$.
Since in particular $k \leq j$, the expression \eqref{SolutionWhenKSmallerThanJ} for the solution still holds, and it will be enough to prove  
%
\begin{equation}\label{Plug4Tris}
\left| \sum_{n_\ell =\lambda^{k-1}}^{\lambda^{k}-1} e^{ 2 \pi i \left( n_\ell  \left( \frac{p_\ell}{q} + \varepsilon_\ell \right)  - \frac{ n_\ell^2}{q} \right)}  \right|
\lesssim 1
\end{equation}
for all $\ell =1, \ldots, d$, because that way we get 
\begin{align}\nonumber
 \left| S_{N}\left( \frac{2 \pi }{q} \right)   
 f_k \left( 2\pi \, \left(\frac{ p}{q} + \varepsilon \right) \right) \right|
& \lesssim 
\lambda^{-k \left( s_{\alpha} + \frac{d}{2} - \delta \right)   } =  \lambda^{- c k } 
\end{align}
as claimed. 
Now, \eqref{Plug4Tris} is an immediate consequence of the Van der Corput first derivative test in Lemma \ref{VDCTest1}. 
For that, we need to bound the distance between $\mathbb{Z}$ and the derivative of the phase, which is $p_\ell/q + \varepsilon_\ell - 2n_\ell/q$. Since
\begin{equation}
\frac14 \leq \frac{p_\ell}{q} \leq \frac12, \qquad |\varepsilon_\ell| \leq \frac{1}{100\lambda^j}, \qquad \frac{2n_\ell}{q} \leq \frac{2}{\kappa \lambda^{j\frac{\alpha}{d+1} - k}} < \frac{2}{\kappa}\,  \frac{1}{\lambda^{\delta j}}, 
\end{equation} 
taking $\lambda$ large enough we may have $|\varepsilon_\ell + 2n_\ell/q| \leq 1/8$, which implies that 
\begin{equation}
\operatorname{dist}\left( \frac{p_\ell}{q} + \varepsilon_\ell - \frac{2n_\ell}{q}, \mathbb Z \right) \geq \frac{1}{8}.
\end{equation}
This, in turn, implies \eqref{Plug4Tris} by Lemma~\ref{VDCTest1}. 

On the other hand, when  $j \left( \frac{\alpha}{d+1} - \delta \right) \leq k \leq j \left( \frac{\alpha}{d+1} + \delta \right)$, 
we use Van der Corput's second derivative test in Lemma~\ref{VDCTest2} to write 
\begin{equation}
\left| \sum_{n_\ell =\lambda^{k-1}}^{\lambda^{k}-1} e^{ 2 \pi i \left( n_\ell  \left( \frac{p_\ell}{q} + \varepsilon_\ell \right) 
 - \frac{ n_\ell^2}{q} \right)}  \right| \lesssim \frac{ \lambda^k - \lambda^{k-1} }{\sqrt{q}}   + \sqrt{q} 
\end{equation}
because the second derivative of the phase in the sum is $-2/q$. Moreover, the above is bounded by
\begin{equation}
\lesssim \frac{\lambda^k}{ \sqrt{q}} + \sqrt{q} \simeq \lambda^{k-j\frac{\alpha}{2(d+1)}} + \lambda^{j\frac{\alpha}{2(d+1)}} \leq 2\, \max\left(\lambda^{k-j\frac{\alpha}{2(d+1)}} , \lambda^{j\frac{\alpha}{2(d+1)}}  \right),
\end{equation}
so
\begin{equation}
\left|  S_N\left( \frac{2\pi}{q}\right) f_k\left( 2\pi \left( \frac{p}{q} + \varepsilon  \right) \right) \right| \lesssim  \lambda^{-k\left( s_\alpha + \frac{d}{2} - \delta \right)}\, \max\left(\lambda^{k-j\frac{\alpha}{2(d+1)}} , \lambda^{j\frac{\alpha}{2(d+1)}}  \right)^d
\end{equation}
In the case the maximum is $\lambda^{k-j\frac{\alpha}{2(d+1)}}$, we bound the above by 
\begin{equation}
\begin{split}
\lambda^{-k\left( s_\alpha + \frac{d}{2} - \delta \right)}\, \lambda^{kd-j\frac{\alpha d}{2(d+1)}} & = \lambda^{-k\left( s_\alpha - \frac{d}{2} - \delta \right)}\,\lambda^{-j\frac{\alpha d}{2(d+1)}} = \lambda^{k\delta }\, \lambda^{-(j-k) \frac{\alpha d}{2(d+1)}} \\
& \leq \lambda^{k\delta} \lambda^{- a\frac{\alpha d}{2(d+1)}k},
\end{split}
\end{equation}
where the last equality holds because
\begin{equation}
k \leq j\left( \frac{\alpha}{d+1} + \delta \right) \quad \Rightarrow \quad (j-k)\left(\frac{\alpha}{d+1}+\delta\right) \geq k\left(1 - \frac{\alpha}{d+1}-\delta\right)
\end{equation}
and 
\[ a =  \frac{1 - \frac{\alpha}{d+1}-\delta}{\frac{\alpha}{d+1} + \delta } >0. \]
Moreover, writing $\tilde{a} = a\frac{\alpha d}{2(d+1)} $ and choosing $\delta$ small 
enough we get $\lambda^{k\delta} \lambda^{- \tilde{a}k} < \lambda^{-k\tilde{a}/2}$.
Otherwise, if the maximum is $\lambda^{j\frac{\alpha}{2(d+1)}}$, using that $j\alpha/(d+1) \leq k + j\delta$, we get
\begin{equation}
\lambda^{-k\left( s_\alpha + \frac{d}{2} - \delta \right)} \, \lambda^{j\frac{\alpha d}{2(d+1)}} \leq \lambda^{-k\left( s_\alpha + \frac{d}{2} - \delta \right)} \, \lambda^{\frac{kd}{2} + \frac{j\delta d }{2}} = \lambda^{-k\left( s_\alpha - \delta \right)} \, \lambda^{ \frac{j\delta d }{2}}.
\end{equation}
Since
\begin{equation}
j\left( \frac{\alpha}{d+1}-\delta \right) \leq k \quad \Rightarrow \quad j \leq b\, k, \qquad \text{ where } \, \,  b = \frac{1}{\frac{\alpha}{d+1}-\delta}>0,
\end{equation}
the above can be bounded by $\lambda^{-k(s_\alpha - \delta - \frac{bd}{2}\delta)} < \lambda^{-k s_\alpha/2}$, if $\delta$ is chosen small enough.

In short, the solution is smaller than either $\lambda^{-k\tilde{a}/2}$ or $\lambda^{-k s_\alpha/2}$, so letting $c = \min(\tilde{a}/2, s_\alpha/2) >0$ we can write  
\begin{equation}
\left|  S_N\left( \frac{2\pi}{q}\right) f_k\left( 2\pi \left( \frac{p}{q} + \varepsilon  \right) \right) \right| \lesssim \lambda^{-ck},
\end{equation}
as we stated in \eqref{abc(i)stronger}.

\subsection*{Proof of \eqref{abc}(iii)}
The points we consider now are
\begin{equation}\nonumber
t = \frac{2 \pi}{q}, \quad x = 2\pi \left( \frac{ p}{q} + \varepsilon \right), \qquad \text{ where } \, \,  
 p \in \mathbb{Z}^d, \, q \in \mathbb{N},\,  \varepsilon \in \mathbb R^d 
\end{equation}
such that $q \equiv 0 \text{ (mod 4)}$,  $q \simeq \lambda^{j\frac{\alpha}{d+1}}$,  $p_\ell \equiv 0 \text{ (mod 2)}$, $q/4 \leq p_\ell \leq q/2$ and 
\[ \frac{1}{200\lambda^j} \leq \varepsilon_\ell \leq \frac{1}{100\lambda^j} \]
for all $\ell = 1, \ldots, d$.

From the definitions of $S_N(t)$ and of $f_k$, we have  $S_N(t) f_k = 0$ whenever $N < \lambda^{k-1}$. Also, $S_N(t) f_k = S_{\lambda^k}(t) f_k$ for every $N \geq \lambda^k$, so we can assume $\lambda^{k-1} \leq N \leq  \lambda^k-1$.
Thus 
\begin{align}\label{Technical}
S_N\left( \frac{2 \pi }{q} \right) f_k \left(  2\pi \left( \frac{ p}{q} + \varepsilon \right) \right) 
& =  
\lambda^{-k \left( s_{\alpha} + \frac{d}{2} - \delta \right) }
\prod_{\ell=1}^d \sum_{n_\ell =\lambda^{k-1}}^{N}  e^{ 2 \pi i \left( n_\ell  \left( \frac{p_\ell}{q} + \varepsilon_\ell \right)  -  \frac{ n_\ell^2}{q}  \right)}  .
\end{align}
We will prove 
\begin{equation}\label{Plug4Bis}
\left|  \sum_{n_\ell =\lambda^{k-1}}^{N} e^{ 2 \pi i \left( n_\ell  \left( \frac{p_\ell}{q} + \varepsilon_\ell \right)  -  n_\ell^2 \frac1q \right)} \right|
\lesssim  \lambda^{j  \left(1- \frac{\alpha}{2(d+1)} \right) },   \qquad \forall k > j, 
\end{equation}
so that
\begin{equation}
\begin{split}
\left| S_N \left( \frac{2 \pi }{q} \right)   f_j \left( 2\pi \left( \frac{ p}{q} + \varepsilon \right) \right) \right|
& \lesssim \lambda^{-k \left( s_{\alpha} + \frac{d}{2} - \delta \right)  } \lambda^{j d \left(1- \frac{\alpha}{2(d+1)} \right)}  \\ 
&  = \lambda^{- (k - j) \left( s_{\alpha} + \frac{d}{2} - \delta \right) } \, \lambda^{-j\left( s_{\alpha} + \frac{d}{2} - \delta - d + \frac{\alpha d}{2(d+1)} \right)}  \\
&  = \lambda^{j \delta  - (k-j) \left( s_{\alpha} + \frac{d}{2} - \delta \right)}  \\
&   \leq \lambda^{j \delta  - c(k-j) }.   
\end{split}
\end{equation}
for any $0 < c < s_{\alpha} + \frac{d}{2} - \delta$, as we wanted to prove. To establish \eqref{Plug4Bis}, we proceed as in the proof of \eqref{abc}(i) and we arrive at the analogue of \eqref{Plug3}, 
\begin{equation}\label{Plug3Bis}
\begin{split}
 \sum_{n_\ell =\lambda^{k-1}}^{N} e^{ 2 \pi i \left( n_\ell \left(  \frac{p_\ell}{q} + \varepsilon_\ell \right)  -  \frac{ n_\ell^2 }{q}  \right)  } & = \sum_{m_{\ell} = 0}^{R_k - L_k -1} e^{2 \pi i (L_k + m_\ell ) q \varepsilon_\ell} \sum_{ r_\ell = 0}^{q - 1}e^{ 2 \pi i \left( r_\ell   \left( \frac{p_\ell}{q} + \varepsilon_\ell \right) -  \frac{r_{\ell}^2}{q}  \right)} \\ 
& +  \sum_{ n_\ell \in J_1 \cup J_2} e^{ 2 \pi i \left( n_\ell  \left( \frac{p_\ell}{q} + \varepsilon_\ell \right) -   \frac{ n_\ell^2 }{q} \right) }
= I + II  ,
\end{split}
\end{equation}
where $L_k$ and $R_k$ are respectively the smallest and largest integers such that 
\begin{equation}\label{DefLk>}
L_k q \geq \lambda^{k-1}, \quad R_k q \leq  N
\end{equation}
and
\begin{equation}\label{DefLk>Bis}
J_1 = \{ \lambda^{k-1}, \ldots, L_k q -1  \}, \quad
J_2 = \{ R_k q,  \ldots, N \},
\end{equation}
which satisfy $|J_1|, |J_2| \leq q$.

When estimating $|I|$, the phases $2\pi i (L_k + m_\ell ) q \varepsilon_\ell$ are not all coherent any more, so we can exploit more cancellations. It is for this that we need to restrict the range of $\varepsilon$ imposing $\varepsilon_\ell \geq \frac{1}{200 \lambda^j}$. On the other hand, since
$\varepsilon_\ell \leq \frac{1}{100 \lambda^j}$, we can use Lemma~\ref{ItRemainsLemma} as before,    
which gives
\begin{align}\nonumber
 |I|   
 & \simeq \sqrt{q} \, 
  \left| \sum_{m_\ell = 0}^{R_k - L_k - 1}  
e^{2 \pi i m_\ell q \varepsilon_\ell} \right|
 = \sqrt{q}  \,   \left|
 \frac{1 -  e^{  2 \pi i  ( R_k - L_k) q \varepsilon_\ell} }{1 - e^{ 2 \pi i q \varepsilon_\ell} } \right|
  \\ \nonumber
  & \quad \quad =   \sqrt{q} \, \left| \frac{\sin \left(  \pi (R_k - L_k)   q \varepsilon_\ell \right)}{\sin \left(\pi   q \varepsilon_\ell \right)} e^{  \pi i \left( R_k - L_k - 1 \right)   q \varepsilon_\ell } \right|
\end{align}
Thus 
\begin{equation}\label{Plug1Bis}
|I| \lesssim  \frac{\sqrt{q}}{|\sin \left(\pi   q \varepsilon_\ell \right)|} \simeq \frac{\sqrt{q}}{q |\varepsilon_\ell|} \simeq \frac{\lambda^{j}}{\sqrt{q}} 
\simeq \lambda^{j  \left(1- \frac{\alpha}{2(d+1)} \right) }  ,
\end{equation}
where we used $\varepsilon_\ell  \geq \frac{1}{200 \lambda^j}$. 

When estimating $|II|$, since $|J_1|, |J_2| \leq q$, the Van der Corput second 
derivative test in Lemma \ref{VDCTest2} gives 
\begin{equation}\label{Plug2Bis}
|II|  \lesssim \sqrt{q} \simeq \lambda^{j  \frac{\alpha}{2(d+1)}  } \leq  \lambda^{j\left(1- \frac{\alpha}{2(d+1)} \right)}  .
\end{equation}
Plugging \eqref{Plug1Bis} and \eqref{Plug2Bis} into \eqref{Plug3Bis}, we obtain \eqref{Plug4Bis}, which concludes the proof of \eqref{abc}(iii). 

\hfill $\Box$

\section{The dimension of $\Gamma$}\label{SECTION_Measure}

This section is devoted to proving that $\operatorname{dim}_\mathcal{H} \Gamma = \alpha$. Let us first describe the points of $\Gamma$, which was defined in \eqref{Gamma}. For convenience, let us write $\tau = (d+1)/\alpha$, so that we will prove that $\operatorname{dim}_\mathcal{H} \Gamma = (d+1)/\tau$.

That $x \in \Gamma$ means that there exists a sequence $(j_k)_{k \in \mathbb{N}} \subset \mathbb N$ such that $x \in \Gamma^{j_k}$ for every $k \in \mathbb N$. This motivates the following chain of equivalences:
\begin{equation}
\begin{split}
x \in \Gamma & \quad \Longleftrightarrow \quad \text{ there exists }  (j_k)_{k \in \mathbb{N}} \subset \mathbb N  \text{ such that }  x \in \Gamma^{j_k}, \, \forall k \\
& \quad \Longleftrightarrow \quad \text{ there exists }  (j_k)_{k \in \mathbb{N}} \subset \mathbb N \text{ such that for all } k  \text{ there exists } t_{j_k} \in T^{j_k} \\
& \quad \qquad \qquad \text{ such that } x \in X^{j_k}_{t_{j_k}} \\
& \quad \Longleftrightarrow \quad \text{ there exists }  (j_k)_{k \in \mathbb{N}} \subset \mathbb N, \,  \exists (q_{j_k})_{k \in \mathbb{N}} \subset \mathbb{N}, \, \exists (p_{j_k})_{k \in \mathbb{N}} \subset \mathbb{Z}^d \\
& \quad \qquad \qquad \text{ such that }  q_{j_k} \equiv 0 \text{ (mod 4}), \, \,  \kappa \lambda^{j_k/\tau} \leq q_{j_k} \leq \lambda^{j_k/\tau},  \\
 & \quad \qquad \qquad \quad (p_{j_k})_i \equiv 0 \text{ (mod 2)}, \quad \frac{q_{j_k}}{4} \leq (p_{j_k})_i \leq \frac{q_{j_k}}{2}, \\
 & \quad \qquad \qquad \quad \quad \text{ and }  \frac{c_1}{\lambda^{j_k}} \leq x_i - \frac{(p_{j_k})_i}{q_{j_k}} \leq \frac{c_2}{\lambda^{j_k}} \quad \text{ for all } i,  k.
\end{split}
\end{equation}
In the previous section the particular values $ c_1 = 1/200$ and $c_2 = 1/100$ were chosen, but they could in general be any $0 < c_1 < c_2 \leq 1$. 
Since $q_{j_k} \simeq \lambda^{j_k/\tau}$, $x$ can be approximated by a rational with an error of $q_{j_k}^{-\tau}$. Moreover, if the intervals $\{ (\kappa\,\lambda^{j/\tau}, \lambda^{j/\tau}) \}_{j \in \mathbb N}$ cover all the real line, then the above can be rewritten as
\begin{equation}\label{DescriptionOfGamma}
\begin{split}
x \in \Gamma  & \Longleftrightarrow  \exists \text{ infinitely many }  (p,q) \in \mathbb{Z}^d \times \mathbb N \text{ such that }\\
& \qquad \quad q \equiv 0 \text{ (mod 4}), \quad  p_i \equiv 0 \text{ (mod 2)}, \quad \frac{q}{4} \leq p_i \leq \frac{q}{2}  \\
 & \qquad \quad \text{and} \quad  \frac{c_1}{q^\tau} \leq x_i - \frac{p_i}{q} \leq \frac{c_2}{q^\tau}, \quad \forall i = 1, \ldots, d, \\
\end{split}
\end{equation}
where $0<c_1<c_2$ may have changed, but are still fixed constants. To be absolutely correct, $\Gamma$ is contained in the set described on the right hand side with some constants, while the right hand side set is included in $\Gamma$ maybe with different constants. However, since the results we obtain are independent of such constants, we may assume that the equivalence \eqref{DescriptionOfGamma} holds.   As was said above, for that we need
\begin{equation}
\kappa \lambda^{j/\tau} < \lambda^{(j-1)/\tau} \, \Longleftrightarrow \, \lambda^{1/\tau} < \frac{1}{\kappa}.  
\end{equation}
Choose then $\lambda = (\kappa^{-1} - 1)^\tau$, which can be done as large as wished by choosing $0 < \kappa < 1$ as small as needed. 

The description of $\Gamma$ in \eqref{DescriptionOfGamma} already suggests that its dimension must be $(d+1)/\tau$, in view of the well-known Jarn\'ik-Besicovitch theorem \cite{Besicovitch1934,Jarnik1931}.
\begin{theorem}[Jarn\'ik-Besicovitch]\label{Theorem_JarnikBesicovitch}
Let $d >0$, $\tau \geq 1 + 1/d$ and 
\begin{equation}
J = \left\{ \,  x \in [0,1]^d \, : \, \sup_{i = 1, \ldots, d}\left| x_i - \frac{p_i}{q} \right| \leq \frac{1}{q^\tau} \, \text{ for infinitely many } (p,q) \in \mathbb{Z}^d\times\mathbb{N} \, \right\}.
\end{equation}
Then, $\operatorname{dim}_{\mathcal{H}} J = (d+1)/\tau$. 
\end{theorem}

We cannot use this theorem directly due to the restrictions on the parity of $p$ and $q$, and also due to the additional lower bound for the error of the approximations. However, we will be able to adapt its proof to our setting by working with
\begin{equation}\label{DefinitionG}
\begin{split}
G = \Bigg\{ \, x \in [0,1]^d \, &  : \,  \frac{c_1}{q^\tau} \leq x_i - \frac{p_i}{q} \leq \frac{c_2}{q^\tau} \quad \forall i = 1, \ldots, d \, \text{ for infinitely many }  \\
& \qquad  (p,q) \in \mathbb{N}^d\times \mathbb N \text{ with } p_i \in \left(\frac{q}{8},\frac{q}{4}\right)\, \,  \forall i = 1,\ldots, d   \,  \Bigg\}.
\end{split}
\end{equation}
Let $x \in G/2$ so that there exists $y \in G$ such that $x = y/2$. That means that there exist infinitely many $p,q$  with the restrictions in \eqref{DefinitionG} such that 
\begin{equation}\label{ChoiceOfConstants}
\begin{split}
\frac{c_1}{q^\tau} \leq y_i - \frac{p_i}{q} \leq \frac{c_2}{q^\tau}  \quad & \Longleftrightarrow \quad \frac{c_1}{2q^\tau} \leq x_i - \frac{p_i}{2q} \leq \frac{c_2}{2q^\tau} \\
& \Longleftrightarrow \quad \frac{4^\tau c_1}{2(4q)^\tau} \leq x_i - \frac{2p_i}{4q} \leq \frac{4^\tau c_2}{2(4q)^\tau}
\end{split}
\end{equation}
for all $i=1,\ldots,d$. Let $\tilde{q} = 4q$, $\tilde{p} = 2p$, $\tilde{c}_1 = 4^\tau c_1 /2$ and $\tilde{c}_2 = 4^\tau c_2 /2$  so that 
\begin{equation}
\frac{ \tilde{c}_1}{\tilde{q}^\tau} \leq x_i - \frac{\tilde{p}_i}{\tilde{q}} \leq \frac{\tilde{c}_2}{\tilde{q}^\tau} 
\end{equation}
 such that  $\tilde{q} \in 4\mathbb N$, $\tilde{p}_i \in 2\mathbb N$ and $\tilde{p}_i \in (\tilde{q}/4,\tilde{q}/2)$ for all $i = 1, \ldots, d$. These equivalences together with \eqref{DescriptionOfGamma} and a proper choice of the constants $\tilde{c}_1, \tilde{c}_2$ imply that $G/2 = \Gamma$, and consequently, 
\begin{equation}
 \operatorname{dim}_{\mathcal H} \Gamma = \operatorname{dim}_{\mathcal H} G.
\end{equation}
In what remains of this section, we prove that $\operatorname{dim}_{\mathcal H}G = (d+1)/\tau$. 
We will be using the supremum norm and distance,
\begin{equation}
\left| x \right|_\infty = \sup_{i=1,\ldots, d} |x_i|, \qquad \dist_\infty(x,y) = |x-y|_\infty, \qquad \forall x,y \in \mathbb{R}^d.
\end{equation}

\subsection{Upper bound}
The upper bound for the dimension of $G$ follows easily from its definition. Indeed, letting
\begin{equation}
A(x,r_1,r_2) =  x + [r_1,r_2]^d, \qquad \forall x \in \mathbb{R}^d, \, 0 < r_1 < r_2,
\end{equation}
be cubes close to $x$ that fit the definition of $G$ in \eqref{DefinitionG}, we see that
\begin{equation}
G \subset \bigcup_{q = Q_0}^\infty \bigcup_{p \in \mathbb{N}^d \cap [\frac{q}{8},\frac{q}{4}]^d } A\left(\frac{p}{q}, \frac{c_1}{q^\tau}, \frac{c_2}{q^\tau}\right), \qquad \forall Q_0 >0.
\end{equation}
All such cubes have diameter 
\[ \frac{ \sqrt{d}\, (c_2-c_1)}{q^\tau} \lesssim \frac{1}{Q_0^\tau},  \] 
which tends to zero when $Q_0$ grows to infinity. Thus, let $s>0$ so that 
\begin{equation}
\begin{split}
\mathcal{H}^s(G) & \leq  \lim_{Q_0 \to \infty}\sum_{q=Q_0}^\infty \sum_{p \in \mathbb{N}^d \cap  [\frac{q}{8},\frac{q}{4}]^d }  \left(d^\frac12 \, \frac{c_2-c_1}{q^\tau} \right)^s \\
&  = \left(d^\frac12\, (c_2-c_1)\right)^s\, \frac{1}{8^d}\,  \lim_{Q_0 \to \infty} \sum_{q=Q_0}^\infty \frac{q^d}{q^{\tau s}} = 0
\end{split}
\end{equation}
as long as $\tau s - d >1$, which is equivalent to $s > (d+1)/\tau$. This implies that $\operatorname{dim}_{\mathcal{H}} G \leq (d+1)/\tau$. 
\begin{remark}
This is the standard argument used to compute the upper bound for the dimension of $J$. Indeed, denoting the closed cube with center at $x\in \mathbb{R}^s$ and radius $r>0 $ by  
\begin{equation}
B_\infty(x,r) = \{ \, y \in \mathbb{R}^d \, : \,  |x- y |_\infty \leq r \,   \}, 
\end{equation}
it is immediate to check that 
\begin{equation}
J \subset \bigcup_{q = Q_0}^\infty \bigcup_{p \in \mathbb{N}^d \cap [0,q]^d } B_\infty\left(\frac{p}{q}, \frac{1}{q^\tau}\right), \qquad \forall Q_0 >0.
\end{equation}
and that,  as a consequence, $\operatorname{dim}_{\mathcal{H}} J \leq (d+1)/\tau$. 
\end{remark}

\subsection{Lower bound}\label{SUBSECTION_LowerBound}

To give the lower bound of $G$ we will build a Cantor-like set inside $G$, whose dimension will be bounded below by the following standard lemma, which is in turn based on the mass distribution principle. It can be found in \cite[Example 4.6]{Falconer2014}, with a generalization to higher dimensions in \cite[Lemma 2.4]{Durand2015}.
\begin{lemma}\label{LemmaOfDimension}
Let $\{ E_k \}_{k \in \mathbb{N}}$ be a sequence of compact subsets of $\mathbb{R}^d$ satisfying the following properties:
\begin{itemize}
	\item The sequence is nested, this is, $E_{k} \subset E_{k-1}$ for every $k \in \mathbb{N}$. 
	\item $E_k$ is a union of disjoint compact sets, which we name $k$-level sets.
	\item The distance between any two $k$-level sets is at least $\epsilon_k >0$, and this sequence is decreasing.
	\item Each $(k-1)$-level set contains at least $m_k \geq 2$ $k$-level sets.
\end{itemize}
Then, 
\begin{equation}
\operatorname{dim}_{\mathcal H}\left(  \bigcap_{k \in \mathbb{N}} E_k \right) \geq \liminf_{k \to \infty} \frac{\log(m_1\,m_2\, \ldots\,m_{k-1})}{-\log\left(\epsilon_k\,m_k^{1/d}\right)}.
\end{equation}
\end{lemma}

The following lemma allows us to build such a sequence of nested sets fulfilling the reproduction and separation criteria. It corresponds to \cite[Lemma 3.1]{Durand2015}, we reproduce here the proof for completeness. 
\begin{lemma}\label{LemmaOfCubes}
Let $C \subset [0,1]^d$ be a closed cube with side-length $0 < l \leq 1$. Let $n \in \mathbb N$ such that $n \geq 2^{15d}/l^{d+1}$, and also
\begin{equation}
Q_n = \left\{ \, q \in \mathbb N \, : \, 2^{-6d}\, n \leq q \leq n \,  \right\}. 
\end{equation}
Let also $\tau \geq 1 + \frac{1}{d}$. Then, there are at least $2^{-18d^2}\, l^d\, n^{d+1}$ cubes $B_\infty(p/q, 1/q^\tau)$ satisfying 
\begin{itemize}
	\item $q \in Q_n$, 
	\item all are included in $C$, and 
	\item the distance between any two such cubes is at least $n^{-1-1/d}$. 
\end{itemize}
\end{lemma}
\begin{proof}
The Dirichlet approximation theorem states that given $n \in \mathbb N$ and for every $x \in \mathbb{R}^d$, there exist $q \in \mathbb{N}$ and $p \in \mathbb Z^d$ such that $1 \leq q \leq n$ and 
\[  \bigg| x - \frac{p}{q} \bigg|_\infty \leq \frac{1}{qn^{1/d}}. \]
 Apply it to every $x \in C$ and call $q(x)$ the minimal $q$ among all possible outcomes of the theorem. By construction, $q(x) \leq n$.

Let $1 < \beta < n$ split the range of $q(x)$ into two; define $Q_n = \{ q \in \mathbb N \, : \, n/\beta \leq q \leq n \}$ so that $q(x)$ is of the order of $n$. Then, taking 
\[ \frac{\beta}{n^{1+1/d}} < \frac{l}{2},\]
 choose a set of rationals $p/q$ such that 
\begin{itemize}
	\item $ q \in Q_n $,
	\item $ \operatorname{dist}_\infty(p/q, \mathbb R^d \setminus C) > (\beta/n)^{1+1/d}$, and
	\item $ |p/q -  p'/q'|_\infty > 3(\beta/n)^{1+1/d}$ for every two points $p/q$, $p'/q'$ in the set,
\end{itemize}
and let $D_n$ be a maximal set among all those, so that there is no set of rationals satisfying the above conditions with a cardinality larger than that of $D_n$. Thus, it is immediate to check the cubes in the set
\[ \left\{  B_\infty\left(\frac{p}{q}, \frac{1}{q^\tau} \right)\, : \, \frac{p}{q} \in D_n  \right\} \] 
satisfy the following properties:
\begin{itemize}
	\item $B_\infty (p/q,1/q^\tau) \subset C$ for every $p/q \in D_n$, 
	\item All such cubes are separated by a distance at least $n^{-1-1/d}$.
\end{itemize}
Observing that the radius of the cubes is $q^{-\tau}$, these properties are the result of combining $q\in Q_n$ yielding 
\[ \frac{1}{q^{\tau}} \leq \left(\frac{\beta}{n} \right)^\tau \leq \left( \frac{\beta}{n} \right)^{1+1/d} \]
 with the other two properties of $ D_n$. Hence, it only remains to compute the cardinality of $D_n$.

The goal will be to see that points $x$ with $q(x) \in Q_n$ can be covered by balls centered in $D_n$ and with radius of the order of $n^{-1-1/d}$. For that, define 
\begin{equation}
C_2 = \{ x \in C \, : \, q(x) < n/\beta \}
\end{equation}
If $x \notin C_2$, then $n/\beta \leq q(x) \leq n$, which means that it can be approximated by 
\[ \left| x - \frac{p}{q(x)} \right|_\infty \leq \frac{\beta}{n^{1+1/d}}. \]
 We want to find some $p'/q' \in D_n$ close enough to $p/q(x)$, but that may not be possible if $x$ is too close to the border of $C$. Hence, let $\alpha >0$ and define
\begin{equation}
C_1 = \left\{ \, x \in C \, : \, \operatorname{dist}_\infty(x,\mathbb R^d \setminus C) > \frac{\alpha}{  n^{1+1/d}}  \,   \right\},
\end{equation}
for which we need $\alpha/n^{1+1/d} < l/2$. 
Then, if $x \in C_1 \setminus C_2$,
\begin{equation}
\operatorname{dist}_\infty\left(\frac{p}{q(x)},\mathbb R^d \setminus C \right) > \frac{\alpha - \beta}{ n^{1+1/d}}, 
\end{equation} 
so we also need $\alpha > \beta$. Moreover, since we want that $p/q(x)$ is in the range of the rationals in $D_n$, we will ask $\alpha - \beta > \beta^{1+1/d}$. That way, $p/q(x)$ is a candidate to be in $D_n$ and two possibilities arise:
\begin{enumerate}
	\item Either $p/q(x) \in D_n$, or
	\item $p/q(x) \notin D_n$, in which case there exists $p'/q' \in D_n$ such that 
	\begin{equation}
	\Big| \frac{p}{q(x)} - \frac{p'}{q'} \Big|_\infty \leq 3 \left( \frac{\beta}{n}\right)^{1+1/d},
	\end{equation} 
	because the contrary would mean that all $p'/q' \in D_n$ are further than $3(\beta/n)^{1+1/d}$ from $p/q(x)$, and then $D_n$ is not maximal because $p/q(x)$ could be in $D_n$ but it is not. 
\end{enumerate} 
In any of the two cases, there exists $p'/q' \in D_n$ closer than $3(\beta/n)^{1+1/d}$  from $p/q(x)$, which by the triangle inequality implies that 
\begin{equation}
\left| x - \frac{p'}{q'} \right|_\infty \leq \left| x - \frac{p}{q(x)} \right|_\infty + \left| \frac{p}{q(x)} - \frac{p'}{q'} \right|_\infty  \leq \frac{\beta+ 3 \beta^{1+1/d}}{n^{1+1/d}} \leq \frac{3\alpha}{n^{1+1/d}}. 
\end{equation} 
In other words, 
\begin{equation}
C_1 \setminus C_2 \subset \bigcup_{p/q \in D_n} B_\infty\left( \frac{p}{q}, \frac{3\alpha}{n^{1+1/d}}  \right),
\end{equation}
so computing the volume we get
\begin{equation}
\mathcal L^d (C_1) - \mathcal L^d(C_2) \leq \# D_n \frac{6^d\alpha^d}{n^{d+1}}.
\end{equation}
Since $C_1$ is a cube of side-length $l - 2\alpha /  n^{1+1/d}$, we get $\mathcal L^d (C_1) = (l-2\alpha /n^{1+1/d})^d$. For an upper bound of the measure of $C_2$, let $x \in C_2$ so that $q(x) < n/\beta$ and 
\begin{equation}
\left| x - \frac{p}{q(x)}  \right|_\infty \leq \frac{1}{qn^{1/d}} < \frac{1}{q}.
\end{equation}
for some $p \in \mathbb Z^d$. This means that $p/q$ cannot be further than $1/q$ from $C$, since the contrary would imply that $x \notin C$. Thus, 
\begin{equation}
C_2 \subset \bigcup_{\substack{q < n/\beta \\ \operatorname{dist}_\infty(p/q,C) < 1/q }} B_\infty\left( \frac{p}{q}, \frac{1}{qn^{1/d}} \right).
\end{equation}
For each fixed $q < n/\beta$, the $\ell^\infty$ distance between successive centers of such balls is $1/q$, so there are at most 
\[ \left( \frac{l}{1/q} + 3 \right)^d = (lq + 3)^d \]
 of them. Hence, 
\begin{equation}
\begin{split}
\mathcal L^d(C_2) & \leq \sum_{q =1}^{ n/\beta} (lq+3)^d\,\frac{2^d}{q^dn} = \frac{2^d}{n}\, \sum_{q =1}^{ n/\beta} \left(l+\frac{3}{q}\right)^d \leq \frac{2^d}{n}\, \left( \sum_{q=1}^{1/l}   \frac{4^d}{q}    + \sum_{q=1/l}^{n/\beta}     4^d l^d \right) \\
&  \leq \frac{8^{d}}{n}\, \left( \frac{1}{l} + \frac{nl^d}{\beta} \right) = 8^{d}\, l^d\, \left( \frac{1}{nl^{d+1}} + \frac{1}{\beta}  \right).
\end{split}
\end{equation}
Consequently, 
\begin{equation}
\#D_n \geq \frac{l^d\, n^{d+1}}{6^d\alpha^d} \left(  \left( 1 - \frac{2\alpha}{l\, n^{1+1/d}} \right)^d - 8^d \left( \frac{1}{n\, l^{d+1}} + \frac{1}{\beta} \right)  \right).
\end{equation}
So that the right hand side is of the order of $l^d\, n^{d+1}$, we need the quantity inside the parentheses to be positive. For that take 
\begin{equation}
\frac{2\alpha}{l\, n^{1+1/d}} < \frac12  \quad \Longleftrightarrow \quad l\, n^{\frac{d+1}{d}} \geq 4\alpha
\end{equation}
and 
\begin{equation}
\frac{1}{n\, l^{d+1}} \leq \frac{1}{2^{4d+2}} \, \text{ and } \, \frac{1}{\beta} \leq \frac{1}{2^{4d+2}} \quad \Longleftrightarrow \quad n\,l^{d+1} \geq 2^{4d+2} \, \text{ and } \, \beta \geq 2^{4d+2},
\end{equation}
so that $\#D_n \geq l^d\, n^{d+1} / (6^d 2^{d+1}\alpha^d)$. Observe now that $nl^{d+1} < ln^{1+1/d}$ because $l \leq 1 < n$, so conditions
\begin{equation}
n\, l^{d+1} \geq \max(2^{4d+2}, 4\alpha), \quad \beta \geq 2^{4d+2}, \quad \alpha > \beta + \beta^{1+1/d}
\end{equation}
suffice, and any choice of $\alpha, \beta$ as above is valid. The numbers in the statement correspond to $\beta = 2^{6d}$ and $\alpha = 2^{13d}$ and the condition $n\, l^{d+1} \geq 2^{15d} \geq 4\alpha$, such that $6^d2^{d+1}\alpha^d \leq 2^{4d+1} 2^{13d^2} \leq 2^{18d^2}$ and $\#D_n \geq 2^{-18d^2}l^dn^{d+1}$.
\end{proof}

Once we have Lemmas~\ref{LemmaOfDimension} and \ref{LemmaOfCubes}, let us build the Cantor-like set inside $G$. 
Define 
\begin{equation}
H_n = \bigcup_{q \in Q_n} \bigcup_{p \in \mathbb{N}^d \cap  (\frac{q}{8}, \frac{q}{4})^d} B_\infty\left(\frac{p}{q},\frac{1}{q^\tau}\right), \qquad \widetilde{H}_n = \bigcup_{q \in Q_n} \bigcup_{p \in \mathbb{N}^d \cap  (\frac{q}{8}, \frac{q}{4})^d} A\left(\frac{p}{q},\frac{c_1}{q^\tau}, \frac{c_2}{q^\tau}\right).
\end{equation} 
The sets $H_n$ are used in the proof for the classical set $J$, while $\widetilde{H}_n$ are the adaptations for $G$. 
Since 
\[ A\left(\frac{p}{q},\frac{c_1}{q^\tau}, \frac{c_2}{q^\tau}\right) \subset B_\infty\left(\frac{p}{q},\frac{1}{q^\tau}\right), \]
 there is an obvious one-to-one correspondence between all such cubes, indexed by the rationals $p/q$.
 
We build the sets $E_k$ iteratively starting with $E_0 = [1/8,1/4]^d$, which has side-length  $ l =1/8$. According to Lemma~\ref{LemmaOfCubes}, choose $n_1 \geq 2^{15d}8^{d+1} = 2^{18d+3}$ so that there are at least $2^{-18d^2}n_1^{d+1}/8^d$ cubes in $H_{n_1}$ inside $E_0$ and they are separated by at least $n_1^{-1-1/d}$. Inside each of these cubes, there is one and only one cube of $\widetilde{H}_{n_1}$, so we build $E_1$ as the union of all such latter cubes such that
\begin{itemize}
	\item they are all inside $E_0$, 
	\item there are at least $m_1 = 2^{-18d^2}n_1^{d+1}/8^d$ of them, and
	\item since they are inside the cubes of $H_{n_1}$, the separation among them is at least $\epsilon_1 = n_1^{-1-1/d}$. 
\end{itemize}
Since we need $m_1 \geq 2$, we ask that $n_1 \geq 2^{\frac{18d^2+3d+1}{d+1}}$ too, so we set
\begin{equation}
 n_1 \geq \max \left\{ 2^{18d+3}, 2^{\frac{18d^2+3d+1}{d+1}} \right\}.
\end{equation}

Suppose now that $E_{k-1}$ is already built by cubes in $\widetilde{H}_{n_{k-1}}$, all of them disjoint and separated by $\epsilon_{k-1}$. Let $C$ be any of these cubes, which is indexed by some $p/q$ with $q \in Q_{n_{k-1}}$ and which consequently has side-length
\begin{equation}
 \frac{c_2-c_1}{n_{k-1}^\tau}    \leq  l = \frac{c_2 - c_1}{q^\tau}  \leq 2^{6d\tau}\,\frac{c_2-c_1}{n_{k-1}^\tau}.
\end{equation}
We use thus Lemma~\ref{LemmaOfCubes} by taking
\begin{equation}\label{ConditionN1}
n_k \geq \frac{2^{15d}}{ (c_2-c_1)^{d+1}}\, n^{\tau(d+1)}_{k-1}
\end{equation}
so that $n_k  \geq 2^{15d}/l^{d+1}$ is satisfied. Thus, there are at least 
\begin{equation}
2^{-18d^2}\, l^d\, n_k^{d+1} \geq 2^{-18d^2} \, (c_2-c_1)^d \, \frac{n_k^{d+1}}{n_{k-1}^{d \tau}}
\end{equation}
cubes in $H_{n_k}$ that lie inside $C$ and that are separated by at least $\epsilon_k = n_k^{-1-1/d}$. Choose the corresponding same number of cubes of $\widetilde{H}_{n_k}$, which are 
\begin{itemize}
	\item all inside $C$, 
	\item at least $m_k = 2^{-18d^2} \, (c_2-c_1)^d \,  n_k^{d+1}/ n_{k-1}^{d \tau}$, and 
	\item separated by at least $\epsilon_k = n_k^{-1-1/d}$. 
\end{itemize}
Then, build $E_k$ as the union of all such cubes coming from all cubes $C$ that make up $E_{k-1}$. Since we need $m_k \geq 2$, we also ask that
\begin{equation}\label{ConditionN2}
n_k \geq \frac{2^{\frac{18d^2+1}{d+1}}}{(c_2-c_1)^\frac{d}{d+1}}\, n_{k-1}^{\frac{d\tau}{d+1}}. 
\end{equation}

Let finally $F = \bigcap_{k \in \mathbb N} E_k$ and let $(n_k)_{k\in\mathbb{N}}$ be a sequence satisfying \eqref{ConditionN1} and \eqref{ConditionN2}. Also, $m_1 = D\, n_1^{d+1}$ and $m_k = D\, n_k^{d+1}/n_{k-1}^{d\tau}$ for every $k > 1$ and for some constant $D>0$ depending only on $d$, and $\epsilon_k = n_k^{-1-1/d}$, which is a decreasing sequence. Thus, 
\begin{equation}
m_1\, m_2\, \ldots\, m_{k-1} = c\, \widetilde{D}^k\, \frac{n_{k-1}^{d+1}}{(n_{k-2}\, n_{k-3}\, \ldots\, n_2 \, n_1)^{d(\tau - 1 - 1/d)}}
\end{equation}
for some constant $c >0$. Since $n_k$ can be taken as large as wanted or needed, impose the extra condition $n_k \geq n_{k-1}^k$, so that when $k$ is large 
$n_k$ grows like $e^{k!}$. Thus, 
\begin{equation}
m_1\, m_2\, \ldots\, m_{k-1}  \simeq n_{k-1}^{d+1}, \qquad \text{ when } k \gg 1.
\end{equation}  
On the other hand, $\epsilon_k\, m_k^{1/d} = D\, n_{k-1}^{-\tau}$, so $- \log (\epsilon_k\, m_k^{1/d}) \simeq \log n_{k-1}^\tau$, and thus by Lemma~\ref{LemmaOfDimension},
\begin{equation}
\operatorname{dim}_{\mathcal{H}} F \geq \liminf_{k \to \infty} \frac{\log{n_{k-1}^{d+1}}}{ \log n_{k-1}^\tau} = \frac{d+1}{\tau}. 
\end{equation}
Now, $x \in F$ means that $x \in E_k \subset \widetilde{H}_{n_k}$ for all $k \in \mathbb{N}$. Hence, there exist $q_k \in Q_{n_k}$ and $p_k \in \mathbb{N}^d \cap (q_k/8,q_k/4)^d$ such that $x \in A(p_k/q_k, c_1/q_k^\tau,c_2/q_k^\tau)$ for every $k \in \mathbb{N}$, which implies that $x \in G$. Thus, we conclude that
\begin{equation}
\operatorname{dim}_{\mathcal H} G \geq \operatorname{dim}_{\mathcal H} F \geq \frac{d+1}{\tau}. 
\end{equation}

\begin{remark}
The proof of the Jarn\'ik-Besicovitch Theorem, Theorem~\ref{Theorem_JarnikBesicovitch} follows exactly the same arguments above, save that $E_0=[0,1]^d$ and that when building the $k$-level sets from a $(k-1)$-level cube in $E_{k-1}$, we are left with the cubes of $H_{n_k}$ given by Lemma~\ref{LemmaOfCubes}. The cubes $A(p/q,c_1/q^\tau,c_2/q^\tau)$ and the corresponding sets $\widetilde{H}_{n_k}$ play no role. 
\end{remark}

\subsection{The case $\boldsymbol{\alpha = d}$}
In Remark~\ref{RemarkOnPositiveMeasureOfGamma}, we explained that when $\alpha < d$ the result $\dim{_\mathcal H} \Gamma = \alpha$ immediately self-improves, since we can easily find an alternative $\Gamma'$ such that $\mathcal H^\alpha (\Gamma') = \infty$. However, this reasoning does not work when $\alpha = d$, so we explicitly included that $\mathcal H^d(\Gamma) >0$ in the statement of Theorem~\ref{MainThm}. We show that here. 

The definition of $\Gamma$ in \eqref{Gamma} shows that it is the intersection of the nested sequence of the sets $\left( \bigcup_{j > n}\Gamma^j \right)_{n \in \mathbb{N}}$, so 
\begin{equation}\label{LowerBoundForLebesgueMeasureOfGamma}
\mathcal H^d (\Gamma) = \lim_{n \to \infty} \mathcal H^d\Bigg( \bigcup_{j > n}\Gamma^j \Bigg) \geq \lim_{n \to \infty} \mathcal H^d (\Gamma^n).
\end{equation}
Thus, let us work with the sets $\Gamma^j$ for $j$ as large as needed. Recalling the definitions \eqref{DefT_j}, \eqref{Def:X_j} and \eqref{Def:Gammaj} and ignoring the $2\pi$ factor, 
\begin{equation}
\begin{split}
x \in \Gamma^j \quad \Longleftrightarrow \quad & \text{ there exists } q \in \left[ \kappa \lambda^{j\frac{d}{d+1}}, \lambda^{j\frac{d}{d+1}}  \right], \quad q \equiv 0 \text{ (mod 4)}, \\
& \quad  \text{ and } p \in  [q/4,q/2]^d, \quad  p_i \equiv 0 \text{ (mod 2)}, \, \, i = 1, \ldots, d \\
& \quad \quad  \text{ such that } \quad   \frac{1}{100\lambda^j} \leq x_i - \frac{p_i}{q} \leq \frac{1}{200\lambda^j}, \quad i = 1, \ldots, d.
\end{split}
\end{equation}
Now, proceeding as in the beginning of the section, if we define $G^j$ as 
\begin{equation}
\begin{split}
x \in G^j \quad \Longleftrightarrow \quad & \text{ there exists } q \in \left[ \frac{\kappa}{4}\,  \lambda^{j\frac{d}{d+1}}, \frac14\, \lambda^{j\frac{d}{d+1}}  \right]  \text{ and } p \in  \left[ \frac{q}{8}, \frac{q}{4} \right]^d \\ 
& \text{ such that }  \quad \frac{c_1}{q^\tau} \leq x_i - \frac{p_i}{q} \leq \frac{c_2}{q^\tau}, \quad i = 1, \ldots, d,
\end{split}
\end{equation}
then $ G^j/2 \subset \Gamma^j$ for a proper choice of the constants $0 < c_1 < c_2$ as in \eqref{ChoiceOfConstants}. 
We remind that in this case $\tau = (d+1)/d$.
Similar to $\widetilde{H}_n$, $G^j$ is nothing but
\begin{equation}
G^j = \bigcup_{q \in \left[ \frac{\kappa}{4}\,  \lambda^{j\frac{d}{d+1}}, \frac14\, \lambda^{j\frac{d}{d+1}}  \right]}\, \bigcup_{ p \in  [\frac{q}{8},\frac{q}{4}]^d } A\left( \frac{p}{q}, \frac{c_1}{q^\tau}, \frac{c_2}{q^\tau} \right).
\end{equation}
In particular, these cubes correspond to those that are in $[1/8,1/4]^d$. 

Now, we use Lemma~\ref{LemmaOfCubes} with $C = [1/8,1/4]^d$, $l = 1/8$ and $n = \lambda^{j\frac{d}{d+1}}/4$. For that, we have to ask first 
\begin{equation}\label{ConditionForJ}
\frac14 \, \lambda^{j\frac{d}{d+1}} > 2^{15d}\,4^{d+1}  \Longleftrightarrow j > \frac{17d^2 + 21d + 4}{d}\, \frac{\ln 2}{\ln \lambda}.
\end{equation}
Then, there are at least $2^{-18d^2}\, 8^{-d} \, \lambda^{jd}$ cubes $B_\infty(p/q,1/q^\tau)$ inside $C$ that satisfy $2^{-6d}\,\lambda^{j\frac{d}{d+1}}/4 \leq q \leq \lambda^{j\frac{d}{d+1}}/4 $ and that are disjoint. Taking $\kappa \leq 2^{-6d}$, they satisfy 
\begin{equation}
\frac{\kappa}{4}\,\lambda^{j\frac{d}{d+1}} \leq q \leq \frac{1}{4}\, \lambda^{j\frac{d}{d+1}}.
\end{equation}
Now, at least the same number of cubes $A(p/q, c_1/q^\tau, c_2/q^\tau)$ in $[1/8,1/4]^d$ are disjoint, which means that at least $2^{-18d^2-3d}\,\lambda^{jd}$ of the cubes that form $G^j$ are disjoint. Hence, 
\begin{equation}
\mathcal H^d(G^j) \geq 2^{-18d^2-3d}\,\lambda^{jd} \, \left(\frac{c_2-c_1}{q^\tau}\right)^d  \simeq (c_2-c_1)^d\, 2^{-18d^2-3d} > 0, 
\end{equation}
which in turn implies 
\begin{equation}
\mathcal{H}^d (\Gamma^j) \geq \frac{1}{2^d}\, \mathcal{H}^d(G^j) \geq (c_2-c_1)^d\, 2^{-18d^2-4d} > 0. 
\end{equation}
This holds for every $j$ as in \eqref{ConditionForJ}, which together with \eqref{LowerBoundForLebesgueMeasureOfGamma} implies 
\begin{equation}
\mathcal H^d (\Gamma)  \geq  (c_2-c_1)^d\,  2^{-18d^2-4d} > 0,
\end{equation} 
and the proof is complete.

\section{Proof of Theorem \ref{Thm2}}\label{SectionThm2}


The main conclusion of Theorem~\ref{Thm2}, namely the $\alpha$-almost everywhere convergence property for $s > \frac{d}{2(d+2)} (d + 2 - \alpha)$, follows from the maximal estimate \eqref{MaximalEstimate} by a standard procedure using density arguments and Frostman's lemma. We refer the reader to \cite[Appendix B]{BarceloBennettCarberyRogers2011} for the details. Therefore, we focus on proving the maximal estimate. 

 Let us first recall that estimates of the same nature are known for the case when $\mu$ is the Lebesgue measure, as was proved in \cite{MoyuaVega2008} when $d=1$ and in \cite{WangZhang2019} when $d\geq2$, using Strichartz estimates. In \cite{CompaanLucaStaffilani2020}, based on the optimal periodic Strichartz estimates from \cite{BourgainDemeter2015}, the $L^p$ exponent was improved when $d \geq 3$. We gather all these results in the following proposition. 
\begin{proposition}[\cite{MoyuaVega2008, WangZhang2019, CompaanLucaStaffilani2020}]\label{StrichME}
Let $s > \frac{d}{d+2}$. Then
\begin{equation}\label{MaxInTorus}
\left\| \sup_{0 \leq t \leq 1} |e^{it\Delta } f | \right\|_{L^{\frac{2(d+2)}{d}}(\mathbb{T}^d)} \lesssim \| f \|_{H^{s}(\mathbb{T}^d)} \, .
\end{equation}\end{proposition}

We show in the following proposition that the above estimates for the Lebesgue measure can be used to deduce similar estimates for a general $\alpha$-dimensional measure.
\begin{proposition}\label{PropTransference}
Let $d\ge 1$, $p \geq 1$ and $s_0 \geq 0$. If
\begin{equation}\label{what}
\left\| \sup_{0 < t < 1} |e^{it\Delta}f | \right\|_{L^{p}(\mathbb T^d)} \lesssim
 \| f \|_{H^{s}(\T^d)}, \quad \forall s > s_0
\end{equation}
then
\begin{equation}\label{AMIThmPrelim}
\left\| \sup_{0 < t < 1} | e^{it \Delta} f | \right\|_{L^{p}(\mathbb T^d, d \mu)} \lesssim
c_{\alpha}(\mu)^{\frac{1}{p}}  \| f \|_{H^{s}(\T^d)}, \quad \forall s > \frac{d-\alpha}{p}  + s_{0}  
\end{equation} 
for every $\alpha$-dimensional measure
$\mu$. \end{proposition}



Once we prove this fact, since Proposition~\ref{StrichME} shows that \eqref{what} holds with $p = 2(d+2)/d$ and $s_0=d/(d+2)$, we get 
\begin{equation}
\left\| \sup_{0 < t < 1} | e^{it \Delta} f | \right\|_{L^{\frac{2(d+2)}{d}}(d \mu)} \lesssim
c_{\alpha}(\mu)^{\frac{d}{2(d+2)}} \,  \| f \|_{H^{s}(\T^d)},  
\end{equation}
for every $\alpha$-dimensional measure $\mu$ and whenever
\begin{equation}
s > \frac{d(d-\alpha)}{2(d+2)}  + \frac{d}{d+2} = \frac{d}{2(d+2)}\left( d+2-\alpha \right),
\end{equation}
which is precisely the statement of Theorem~\ref{Thm2}. Hence let us prove Proposition~\ref{PropTransference}. For that, we will use several properties of the Dirichlet kernel \eqref{DirichletKernel} that we give in the following two lemmas.
\begin{lemma}\label{LemmaDirichlet1D}
Let $N \in \mathbb N$. The Dirichlet kernel in one dimension, $d_N$, is symmetric with respect to $\pi$, that is,  $d_N(2\pi - x) = d_N(x)$ for every $0\leq x \leq \pi$, and
\begin{equation}\label{DirichletKErnelBound}
|d_N(x)| \lesssim  \left\{ 
\begin{array}{lll}
N & \mbox{when} & 0 \leq x \leq 1/N \\
1/ |x |  & \mbox{when} & 1/N < x \leq \pi,
\end{array}
\right. 
\end{equation}
or equivalently, $|d_N(x)| \lesssim  \min\left( N, 1/x \right)$. 
\end{lemma}
\begin{proof}
Both properties follow immediately from the well-known formula
\begin{equation}\label{DirichletKernel1D}
d_N(x)  =   \sum_{|k| \leq N} e^{i k x} = \frac{\sin \left( \left( N + 1/2 \right) x\right)}{\sin \left( x/2 \right)}, \qquad x \in \mathbb T.
\end{equation}
\end{proof}

\begin{lemma}\label{LemmaConvolutionWithMu}
Let $N \in \mathbb N$ and $\mu$ be an $\alpha$-dimensional measure. Then,
\begin{equation}
\left(  |D_N| \ast \mu \right) (x) \lesssim c_\alpha(\mu)\, N^{d-\alpha}\, \left( \ln N \right)^d, \qquad \forall x \in \mathbb{T}^d.
\end{equation}
\end{lemma}
\begin{proof}
 Change variables to write 
\begin{equation}\label{MainIntegralDirichl}
\left( \left| D_N  \right| *   \mu \right)(x)  =  \int_{\mathbb T^d} |D_N (x-y)| d \mu(y) = \int_{\mathbb T^d} |D_N (y)| d \mu_x(y),
\end{equation}
where $\mu_x(A) = \mu(x-A)$ for every $A \subset \mathbb{T}^d$. In particular, $\mu_x$ is an $\alpha$-dimensional measure, with the same constant as $\mu$. Now, split the torus into the $2^d$ cubes of length $\pi$ that are obtained by halving each component $\mathbb T$ of the $d$-dimensional torus. For the symmetric property $d_N(2\pi - x_\ell) = d_N(x_\ell)$, the situation for all these cubes is similar to that in $[0,\pi]^d$, so we work in the latter. Divide it into smaller cubes of length $2\pi/N$ indexed as 
\begin{equation}
Q_m = \prod_{\ell=1}^d  \left[2\pi\,\frac{m_\ell}{N},2\pi\, \frac{m_\ell+1}{N}\right], \qquad m = (m_1, \ldots, m_d) \in \{ 0, 1, \ldots, N/2-1 \}^d
\end{equation}
 so that   
\begin{equation}\label{SplittingTheLittleTorus}
\int_{[0,\pi]^d} |D_N (y)| d \mu_x(y) = \sum_{m} \int_{Q_m} |D_N (y)| d \mu_x(y). 
\end{equation}
We may conveniently rewrite the bound in Lemma~\ref{LemmaDirichlet1D} as
\begin{equation}
\left| d_N(y_\ell) \right| \lesssim \frac{N}{ m_\ell  + \delta(y_\ell)}, \quad \text{ where } \delta(y_\ell) = \left\{  \begin{array}{ll}
1, & \text{ if } y_\ell \in \left[ \frac{m_\ell}{N},\frac{m_\ell +1}{N} \right), m_\ell = 0, \\
0, & \text{ if } y_\ell \in \left[ \frac{m_\ell}{N},\frac{m_\ell +1}{N} \right), m_\ell \neq 0,
\end{array}
\right.
\end{equation}
and plug it in \eqref{SplittingTheLittleTorus} so that 
\begin{equation}
\begin{split}
\int_{[0,\pi]^d} |D_N (y)| d \mu_x(y) &  \lesssim \sum_m \left( \prod_{l=1}^d\frac{N}{ m_\ell  + \delta(y_\ell)} \right)\, \mu_x(Q_m) \\
& \lesssim \frac{c_\alpha(\mu) }{N^\alpha} \, \left( \sum_{n=0}^{N/2}\frac{N}{n+\delta(n/N)} \right)^d = \frac{c_\alpha(\mu) }{N^\alpha} \, \left( N + N\sum_{n=1}^{N/2}\frac{1}{n}   \right)^d \\
& \simeq c_\alpha(\mu)\, N^{d-\alpha}\, \left( \ln N \right)^d.
\end{split}
\end{equation}
Since, as noted above, the same bound works for each of the $2^d$ cubes of size $\pi$ in which we split $\mathbb T^d$, we get 
\begin{equation}
\left( \left| D_N  \right| *   \mu \right)(x) \lesssim c_\alpha(\mu)\, N^{d-\alpha}\, \left( \ln N \right)^d, \qquad \forall x \in \mathbb T^d.
\end{equation}
\end{proof}

Let us proceed now to the proof of Proposition~\ref{PropTransference}.
\begin{proof}[Proof of Proposition~\ref{PropTransference}]

By splitting $f$ into its dyadic decomposition $f = \sum_{k =0}^\infty f_k$ where the Fourier coefficients $\widehat{f_k}(n)$ are supported 
in the squared annulus where~$|n_j| \leq 2^k$ for all $j=1, \ldots, d$ and $2^{k-1} \leq |n_j| \leq 2^k$ for at least one $j$,  
it is enough to prove 
\begin{equation}\label{AlphaMaxIneqPrelim}
\left\| \sup_{0 < t <1} | e^{it\Delta} f | \right\|_{L^{p}(\mathbb T^d, d \mu)} \lesssim
c_{\alpha}(\mu)^{\frac{1}{p}} N^{s} \| f \|_{L^{2}(\T^d)},
\end{equation}
whenever $N \in 2^\mathbb{N}$,  
 $ \supp \widehat{f} \subset [-N,N]^d$ 
  and $s > \frac{d-\alpha}{p} + s_{0} $.  For such functions, it is easy to check that the self reproducing formula 
\begin{equation}
e^{it\Delta}f =  D_N * e^{it\Delta}f
\end{equation}  
holds, where $D_N$ is the Dirichlet kernel defined in \eqref{DirichletKernel}.
Letting $p'$ be the H\"older conjugate of $p$ that satisfies $1 = \frac{1}{p} + \frac{1}{p'}$ and using H\"older's inequality, we may write
\begin{align}\label{BeforeFubini1}
 &  \quad \quad \sup_{0 < t <1} |e^{it\Delta} f(x) |  
 \leq    \int_{\T^d} |D_N(x - y)| \sup_{0 < t < 1} |e^{it\Delta}f(y)| \ dy \\ \nonumber
&\qquad \qquad  \leq \left(  \int_{\mathbb T^d}\left| D_N(x-y) \right|\, dy \right)^{\frac{1}{p'}} \, 
\left( \int_{\T^d}  | D_N (x-y) | \sup_{0 < t < 1} 
|e^{it\Delta}f(y)|^{p} \ dy  \right)^{\frac{1}{p}}.
\end{align}
For the first integral, from Lemma~\ref{LemmaDirichlet1D} we get
\begin{equation}\label{IntegralOfDirichletKernel}
\int_{\T^d} | D_N(x-y)|\, dy  = \int_{\T^d} | D_N(y)|\, dy =  \left( \int_{\T} | d_N(y)|\, dy \right)^d  \lesssim (1 + \ln N)^d \simeq (\ln N)^d.
\end{equation}
%
%
%
%
Thus, take in \eqref{BeforeFubini1} the $L^p$ norm with respect to an $\alpha$-dimensional measure $\mu$ so that the fact that the Dirichlet kernel is an even function and Fubini's theorem imply
\begin{equation}\label{IntermediateLpBound}
\begin{split}
\left\| \sup_{0 < t < 1} | e^{it\Delta}f | \right\|^{p}_{L^{p}(\mathbb T^d, d \mu)} & \lesssim (\ln N)^{\frac{dp}{p'}}  \int_{\T^d} \int_{\T^d} \left| D_N (x - y) \right|  \sup_{0 < t < 1} |e^{it\Delta}f(y)|^{p} \ dy\,  d\mu(x) \\
& =  (\ln N)^{\frac{dp}{p'}} \int_{\T^d} \sup_{0 < t < 1} |e^{it\Delta}f(y)|^{p} \left( 
\left| D_N  \right| * \mu  \right)(y) \ dy. 
\end{split}
\end{equation}
Using now Lemma~\ref{LemmaConvolutionWithMu}, we obtain
 \begin{eqnarray}
 \left\| \sup_{0 < t < 1} | e^{it\Delta}f | \right\|^{p}_{L^{p}(\mathbb T^d, d \mu)} & \lesssim & c_{\alpha}(\mu) (\ln N)^{dp} N^{d- \alpha}  
 \left\| \sup_{0 < t <1} | e^{it\Delta}f | \right\|^{p}_{L^{p}(\T^d)}. \nonumber
 \end{eqnarray}
By hypothesis, for every $s > s_0$, the $L^p$ of the maximal function on the right is bounded by $\lVert f \rVert_{H^s(\mathbb T^d)}$, so  we get
\begin{equation}
\begin{split}
\left\| \sup_{0 < t < 1} | e^{it\Delta}f | \right\|_{L^{p}(\mathbb T^d, d \mu)} &  \lesssim c_{\alpha}(\mu)^{1/p} (\ln N)^{d} N^{\frac{d- \alpha}{p}}\,\lVert f \rVert_{H^s(\mathbb T^d)}   \\
&  \simeq c_{\alpha}(\mu)^{1/p} (\ln N)^{d} N^{\frac{d- \alpha}{p} + s  }\,\lVert f\rVert_{L^2(\mathbb T^d)} \\
&  \lesssim c_{\alpha}(\mu)^{1/p}  N^{\frac{d- \alpha}{p} + s + \varepsilon d  }\,\lVert f \rVert_{L^2(\mathbb T^d)} \\
&  = c_{\alpha}(\mu)^{1/p}  N^{\tilde{s}  }\,\lVert f \rVert_{L^2(\mathbb T^d)}
\end{split}
\end{equation}
for any $\tilde{s} = \frac{d- \alpha}{p} + s + \varepsilon d > \frac{d- \alpha}{p} + s_0 + \varepsilon d $ and any $\varepsilon >0$. Thus, taking $\varepsilon$ as small as needed, the above inequality is satisfied for every $\tilde{s} > \frac{d- \alpha}{p} + s_0$, which is \eqref{AlphaMaxIneqPrelim}, what we wanted to prove. 
\end{proof}

\appendix

\section{}\label{Appendix}

As we claimed in the beginning of Section~\ref{SECTION_Setup}, for $s \in (0, d/2]$, $\alpha > d - 2s$ and  $f \in H^{s}(\T^d)$, we show here that the solution $e^{it\Delta}f(x) = \lim_{N \to \infty}S_N(t)f(x)$ exists $\alpha$-almost everywhere.

By a standard density argument combined with Frostman's lemma (for the details we refer again to \cite[Appendix B]{BarceloBennettCarberyRogers2011}) it is enough to prove the following maximal inequality.
\begin{proposition}\label{Carls>0L2}
Let $s \in (0, d/2]$ and $\alpha > d-2s$. Then    
\begin{equation}\label{MaxIneqDefSol}
\left\| \sup_{N> 1} | S_{N}(t)f | \right\|_{L^2(\mathbb T^d, d \mu)} \lesssim
\sqrt{c_{\alpha}(\mu)}  \| f \|_{H^{s}(\T^d)}
\end{equation}
for every $\alpha$-dimensional measure $\mu$.
\end{proposition}

We prove first an auxiliary lemma showing the analogue to \eqref{IntegralOfDirichletKernel}.
\begin{lemma}\label{LastLemma} 
Let $N \in \mathbb N$. Then,  
\begin{equation}\label{DirAgain}
\int_{\T^d} \sup_{M \leq N} | D_M(x) |\, dx  \lesssim (\ln N)^d.
\end{equation}
\end{lemma}
\begin{proof}
The definition of $D_N$ in \eqref{DirichletKernel}  directly gives
\begin{equation}\label{SupDirDec}
\sup_{M \leq N} | D_M (x)| \leq  \prod_{\ell=1}^d \sup_{M \leq N} |d_{M}(x_\ell)|, 
\end{equation}
and by the symmetric property in Lemma~\ref{LemmaDirichlet1D}, 
\begin{equation}
\int_{\T^d} \sup_{M \leq N} | D_M(x) |\, dx  \leq  \left(  \int_{\mathbb T}  \sup_{M \leq N} |d_{M}(x)|\, dx \right)^d = 2^d\,  \left(  \int_0^{\pi}  \sup_{M \leq N} |d_{M}(x)|\, dx \right)^d.
\end{equation}
Now, using $|d_M(x)| \lesssim \min (M, 1/x)$ from Lemma~\ref{LemmaDirichlet1D}, and since $\min (M, 1/x) \leq \min(N,1/x)$ for every $M \leq N$, we deduce that
\begin{equation}\label{SupDirBound}
\sup_{ M \leq N} |d_M(x)|  \lesssim \min(N,1/x).
\end{equation}
Hence, 
\begin{equation}
\int_0^{\pi} \sup_{M \leq N} |d_{M}(x)|\, dx \lesssim \int_0^{1/N} N\, dx + \int_{1/N}^{\pi} \frac{dx}{x} \simeq \ln N
\end{equation}
which concludes the proof. 

 
\end{proof}

We proceed to the proof of Proposition~\ref{Carls>0L2}.
\begin{proof}[Proof of Proposition~\ref{Carls>0L2}]
By the usual Littlewood-Paley decomposition in squared annuli, it is enough to prove
\begin{equation}\label{CarlEasy}
\left\| \sup_{M \leq N} | S_M(t) f | \right\|_{L^{2}(\mathbb T^d, d \mu)} \leq C_{\varepsilon}
\sqrt{c_{\alpha}(\mu)} \, N^{\frac{d-\alpha}{2} + \varepsilon}\,  \| f \|_{L^{2}(\T^d)}, 
\end{equation}
for every $\varepsilon >0$, $N \in 2^\mathbb{N}$ and  
$\supp \widehat{f} \subset [-N,N]^d$. 
Observe that the 
restriction $M \leq N$ in the maximal operator in \eqref{CarlEasy} is due to the fact that $S_M(t)f(x) = S_N(t)f(x)$ for every $M \geq N$.
It is easy to check that the self-reproducing formula 
\begin{equation}
S_M(t)f =  D_M * e^{it \Delta} f
\end{equation}
holds, where $D_N$ is the Dirichlet kernel defined in \eqref{DirichletKernel}. By H\"older's inequality, 
\begin{equation}
\begin{split}
& \sup_{M \leq N} |S_M(t) f(x) |    \leq    \int_{\T^d}   \sup_{M \leq N} |D_M(x - y)|  |e^{it\Delta}f(y)| \ dy \\
& \qquad \qquad \leq  \left( \int_{\mathbb T^d}  \sup_{M \leq N} \left|D_M(x-y) \right| \,  dy \right)^\frac12   \left( \int_{\T^d}  \sup_{M \leq N} | D_M (x-y) |  |e^{it\Delta}f(y)|^{2} \ dy  \right)^{\frac{1}{2}}
\end{split}
\end{equation}
The first term can be bounded using Lemma~\ref{LastLemma}.
Thus, as we did in \eqref{BeforeFubini1}, the fact that the Dirichlet kernel is an even function and Fubini's theorem imply
\begin{equation}\label{BoundOfNorm}
\begin{split}
\left\| \sup_{M \leq N} |S_M(t) f |  \right\|^{2}_{L^{2}(d \mu)} & \lesssim (\ln N)^d \int_{\T^d} \int_{\T^d} \sup_{M \leq N} \left|   D_N (x - y) \right|   |e^{it\Delta}f(y)|^{2} \ dy \,  d\mu(x) \\ 
& =  (\ln N)^d \int_{\T^d}  |e^{it\Delta}f(y)|^2 \left( \sup_{M \leq N} \left| D_N  \right| *   \mu  \right)(y) \ dy. 
\end{split}
\end{equation}
In the fashion of Lemma~\ref{LemmaConvolutionWithMu}, one can prove 
\begin{equation}\label{VeryLast}
\left( \sup_{M \leq N} \left|    D_N  \right| *   \mu \right)(y)   \lesssim  c_{\alpha}(\mu) \, N^{d-\alpha}\,  (\ln N)^d, \qquad \forall y \in \mathbb{T}^d,
\end{equation}
where we need to use \eqref{SupDirDec} and replace \eqref{DirichletKErnelBound} by its maximal analogue \eqref{SupDirBound}. Hence, from \eqref{BoundOfNorm} and using Plancherel's theorem we get
\begin{equation}
\begin{split}
\left\| \sup_{M \leq N} |S_M(t) f |  \right\|_{L^{2}(d \mu)} & \lesssim \sqrt{c_{\alpha}(\mu)} \, (\ln N)^{d} \, N^{\frac{d- \alpha}{2}}  \left\|  e^{it\Delta}f  \right\|_{L^{2}(\T^d)} \\
 & = \sqrt{ c_{\alpha}(\mu)} \, (\ln N)^{d} \, N^{\frac{d- \alpha}{2}} \| f \|_{L^2(\T^d)}.
\end{split}
\end{equation}
Since for any given $\varepsilon>0$ we have $\ln N < N^\varepsilon$ if $N$ is large enough, we conclude
\begin{equation}
\left\| \sup_{M \leq N} |S_M(t) f |  \right\|_{L^{2}(\mathbb T^d, d \mu)} \lesssim \sqrt{ c_{\alpha}(\mu)} \,  N^{\frac{d- \alpha}{2} + \varepsilon d} \, \| f \|_{L^2(\T^d)}
\end{equation}
which is what we wanted to prove in \eqref{CarlEasy}.
\end{proof}

\subsection*{Acknowledgements}
The authors would like to thank Luis Vega for pointing out the counterexample from \cite{MoyuaVega2008}, which is fundamental 
to proving Theorem \ref{MainThm}. R. Luc\`a would also like to thank Ciprian Demeter and Anne Lonjou for useful conversations and Keith Rogers 
for having introduced him to the problem and for having shared many ideas about Proposition~\ref{PropTransference}.

The authors are supported by the Basque Government under program BCAM-BERC 2018-2021 and by the Spanish Ministry of Science, Innovation and Universities under the BCAM Severo Ochoa accreditation SEV-2017-0718. 

D. Eceizabarrena is supported as well by the Spanish Ministry of Education, Culture and Sports under grant FPU15/03078 - Formaci\'on de Profesorado Universitario, the ERCEA under the Advanced Grant 2014 669689 - HADE and the Simons Foundation Collaboration Grant on Wave
Turbulence (Nahmod's Award ID 651469). Part of this work was done while he worked at BCAM - Basque Center for Applied Mathematics.

R. Luc\`a is also supported by IHAIP project PGC2018-094528-B-I00 (AEI/FEDER, UE).

\bibliographystyle{acm}
\bibliography{DaniRenato}

\end{document}